\documentclass[reqno]{amsart}

\pdfoutput=1

\usepackage{amsthm,amsmath,amssymb,amscd}

\usepackage[utf8]{inputenc}
\usepackage[english]{babel}
\usepackage{graphicx}
\usepackage{cite}
 \usepackage{tikz}
\usepackage{enumerate}
\usepackage{hyperref}

\usepackage{xcolor}
\hypersetup{ colorlinks, linkcolor={red!50!black}, citecolor={blue!50!black}, urlcolor={blue!80!black} }

\calclayout

\theoremstyle{plain}
  \newtheorem{theorem}{\bf Theorem}
  \newtheorem{proposition}[theorem]{\bf Proposition}
  \newtheorem{lemma}[theorem]{\bf Lemma}

\theoremstyle{remark}
  \newtheorem{remark}{\bf Remark}

\newcommand{\RR}{\mathbb{R}} 
\newcommand{\NN}{\mathbb{N}} 

\newcommand{\mb}{\mathbf} 

\newcommand{\SP}{\mathbb{S}} 

\newcommand{\norm}[1]{\lVert#1\rVert} 

\newcommand{\supp}{\mathrm{supp \hspace{.1 mm}}} 

\newcommand{\jp}[1]{<#1>} 

\newcommand{\pv}{{\it P.V.}} 

\newcommand{\ds}[1]{d\sigma_{#1\eta}} 

\begin{document}

 
\title[Recovery of  singularities  from Backscattering data]{Recovery of the singularities of a potential from Backscattering data in general dimension}
\author{Cristóbal J. Meroño }
\thanks{Cristóbal J.  Meroño, Departamento de Matemática e Informática aplicadas a la Ingeniería Civil y Naval, Universidad Politécnica de Madrid. E-mail: cj.merono@upm.es}
\date{\today}

\begin{abstract}
We prove that in dimension $n\ge 2$ the main singularities of a complex
potential $q$ having a certain a priori regularity are contained in the
Born approximation $q_{B}$ constructed from backscattering data. This
is archived using a new explicit formula for the multiple dispersion
operators in the Fourier transform side. We also show that ${q-q_{B}}$
can be up to one derivative more regular than $q$ in the Sobolev scale.
On the other hand, we construct counterexamples showing that in general
it is not possible to have more than one derivative gain, sometimes even
strictly less, depending on the a priori regularity of $q$.
\end{abstract}


\maketitle


\section{Introduction and main theorems}

The central problem in inverse scattering for the Schrödinger equation is to recover a potential $q(x)$, $x \in \RR^n$, from the scattering data, the so called scattering amplitude $u_\infty$. The scattering amplitude measures the far field response of the Hamiltonian $H:=-\Delta + q$ to incident plane waves. In backscattering, as the name suggest, only  the far field response appearing in the opposite direction of the incoming wave is considered, or in other words, only the waves scattered in the opposite direction of the incident wave (the echoes).
The usual reconstruction procedure is to construct the Born approximation $q_B$ of the potential, also an $\RR^n$ function as $q$, from the backscattering data contained in $u_\infty$.  This is the linear approximation to the inverse problem and it is widely used in applications. 

From a mathematical point of view an important question that is not completely answered is to establish how much information does the Born approximation contain about the actual potential $q$. This problem can be approached in different ways. One is to look for uniqueness results, that is, if  $q_B$ is enough to determine $q$ (this problem is still open, see next section for references). 
Motivated by the use of the Born approximation in applications, another approach is to ask how much and what kind of information about $q$ can be obtained just by  looking at $q_B$, that is, in a very immediate way. 
In this sense, in \cite{PSo} it was proposed  that the Born approximation must contain the leading singularities of $q$. Since then, this approach has received great amount of attention in different scattering problems. In the case of backscattering we mention, among others, \cite{GU,esw} for recovery of conormal singularities, \cite{OPS, Re} for recovery of singularities in $2$ dimensions, \cite{RV,RRe} in dimensions $2$ and $3$ and  \cite{BM09,BM09quad} in odd dimension  $n\ge3$.

The main objective of this work is to quantify  as exactly as possible how much more regular than $q$  can $q-q_B$ be in general, depending on the  dimension $n$, and the a priori regularity of the potential $q$ measured in the Sobolev scale. The potential can be complex valued. We provide positive and negative results, which answer this question almost completely, except for potentials in a certain range of the Sobolev scale where there is still a gap between them (see figure \ref{fig.dim2y3}).
To measure the regularity, we introduce  the fractional derivative operator $\jp{D}^{\alpha}$, $\alpha\in \RR $  given by the Fourier symbol $\jp{\xi}^\alpha$ with $\jp{x} = (1+|x|^2)^{1/2}$, and the  weighted Sobolev space $W_\delta^{\alpha,p}(\RR^n)$, $\delta \in \RR$,
$$W_\delta^{\alpha,p}(\RR^n) := \{ f\in \mathcal{S}'(\RR^n) : \norm{\jp{\cdot}^\delta \jp{D}^{\alpha} f}_{L^p(\RR^n)} <\infty\}. $$
We usually use the notation $L_\delta^p(\RR^n) := W^{0,p}_\delta(\RR^n)$ and $W^{\alpha,p}(\RR^n) := W^{\alpha,p}_0(\RR^n)$, also we say that $f\in W^{\alpha,p}_{loc}(\RR^n)$ if $\phi f\in W^{\alpha,p}(\RR^n)$ for every $\phi \in C^\infty_c(\RR^n)$. 

As we shall see in the next section, the Born approximations $q_B$ is related to the potential through the Born series expansion,
\begin{equation*}
 {q}_B  \sim  {q} + \sum_{j=2}^{\infty}{Q_{j}(q)}, 
 \end{equation*}
where $Q_{j}(q)$ are certain multilinear operators  describing the (multiple) dispersion of waves (we use the $\sim$ symbol to avoid claiming anything about convergence yet).  We will call the $Q_{2}$ operator the double dispersion operator of backscattering.  A key guiding principle is that in general  $Q_j(q)$ is expected to be smoother as $j$ grows. We can introduce now the main  theorems of this work. 
\begin{theorem} \label{teo:main1}
Let $n\ge 2$ and $\beta \ge 0$. Assume that $q-q_B \in W_{loc}^{\alpha,2}(\RR^n)$ for every $q\in W^{\beta,2}(\RR^n)$ compactly supported, radial, and real. Then $\alpha$ necessarily satisfies,
\begin{equation} \label{eq:mainrange1}
\alpha \le  \begin{cases}
 \,\, 2\beta - (n-4)/2, \quad  if \quad m \le \beta < (n-2)/2,\\
 \, \, \beta + 1, \hspace{18mm} if \quad  (n-2)/2 \le \beta<\infty ,
         \end{cases} 
\end{equation}
\begin{equation} \label{eq:m}
\text{where} \hspace{8mm} m = (n-4)/2 + 2/(n+1).
\end{equation}
\end{theorem}
\begin{theorem}[Recovery of singularities] \label{teo:main2}
Let $n\ge 2$ and $\beta\ge 0$. Assume that $q \in W^{\beta,2}(\RR^n)$  is  compactly supported. Then $q-q_B \in W^{\alpha,2}(\RR^n)$,  modulo a $C^\infty$ function, if the following condition also holds
\begin{equation} \label{eq:mainrange}
\alpha < \begin{cases}
 \,\, 2\beta - (n-3)/2, \quad if \quad (n-3)/2 <\beta < (n-1)/2,\\
 \, \, \beta + 1, \hspace{18mm} if \quad (n-1)/2 \le \beta<\infty .
         \end{cases} 
\end{equation} 
\end{theorem}

See Figure \ref{fig.dim2y3} for a graphic representation of these results for $n=2$ and $n=4$. 

  {Theorem   \ref{teo:main1}} is the first result giving upper bounds for the
maximum possible regularity that can be obtained from the Born
approximation in backscattering. As we shall see, condition
  {(\ref{eq:mainrange1})} is a consequence of upper bounds for the
regularity of the $Q_{2}$ operator given by   {Theorem   \ref{teo:Q2count}}
below. The main reason we need $\beta \ge m$ and compact support is that
the convergence of the (high frequency) Born series in a Sobolev space
$W^{\alpha ,2}(\mathbb{R}^{n})$ is known only under these assumptions
(see   {Proposition   \ref{prop:convergence}} below). A remarkable consequence
of condition   {(\ref{eq:mainrange1})} is that for $\beta <(n-2)/2$ and
$n>2$ it is not possible to reach the expected gain of one derivative
over the regularity of $q$ (see   {Fig.   \ref{fig.dim2y3}} for the cases
$n=2,4$). In fact, we reach the minimum value of $2/(n+1)$ in
  {(\ref{eq:mainrange1})} for the upper bound of the derivative gain
$\alpha -\beta $ when $\beta =m$, which approaches $0$ as $n$ grows.

  {Theorem   \ref{teo:main2}} is a consequence of new estimates of the
$Q_{j}$ operators for $n\ge 2$ and $j\ge 2$ given in   {Theorem   \ref{teo.Qj}} below. As far as we know, these are the first results of
recovery of singularities for every dimension $n$ in backscattering. We
remark that   {Theorem   \ref{teo:main1}} implies that a one derivative gain
is the best possible result and so, the $1^{-}$ derivative gain in
  {(\ref{eq:mainrange})} is optimal except for the limiting case
$\alpha = \beta +1$.  In
\cite{RRe} it was shown that $q-q_{B}$ is in $W^{\alpha ,2}(
\mathbb{R}^{n})$ (modulo a $C^{\infty }$ function) with $n=2,3$ and
$\alpha <\beta +1/2$. Therefore, in dimension 2 we improve the previous
results for all $\beta \ge 0$ (see   {Fig.   \ref{fig.dim2y3}}), but in
dimension $3$ the result in \cite{RRe} is still the best result
for  low a priori regularity $0\le \beta < 1/2$. Also, a similar result to   {Theorem   \ref{teo:main2}} has been obtained in \cite[Corollary 4.8]{BM09} in odd dimension $n\ge 3$ using a certain modified  Born approximation.

Indeed,   {Theorems   \ref{teo:main1} and \ref{teo:main2}} leave a gap of up to $1/2$ derivative when $\max (m,0) \le \beta <(n-1)/2$ between the
positive and negative results. A similar situation is found in the fixed
angle and full data scattering problems, where analogous results to
  {Theorems   \ref{teo:main1} and \ref{teo:main2}} have been proved in
\cite{fix} (see \cite{BFRV10} for the positive results in the case
of full data scattering). In backscattering, this gap has been partially
closed in dimension $3$ by the mentioned result in \cite{RRe} and
in dimension $2$ in \cite{BFRV13}, where a uniform $1^{-}$
derivative gain has been obtained using a weaker regularity scale than
the Sobolev scale $W^{\alpha ,2}$. We will make more observations about
this problem in the final remarks.

 \begin{figure} 
 \centering
\begin{tikzpicture}[scale=1.5]
 \draw[<->] (3,0) node[below]{$\beta$} -- (0,0) --
(0,2.3) node[left]{$\alpha-\beta$};
  \draw [gray, very thin] (0.5,1) -- (0.5,0);
   \draw[dashed, semithick,red,   domain=0:0.5] plot (\x, {1});
    \draw  [semithick,blue](0,0.5) -- (0.5,1);
 \draw  [semithick,blue] (0.5,1) -- (3,1);
  \draw [dash dot ] (0,0.5) -- (3,.5);
  \node at (1,2) {$n=2$};
 \node [left] at (0,1) { \tiny $1$ };
 \node [left] at (0,0.5) { \tiny $\frac{1}{2}$ };
 \node [below] at (0.5,0) { \tiny $\frac{1}{2}$ };
\end{tikzpicture}
\begin{tikzpicture}[scale=1.5]
1. \draw[<->] (3,0) node[below]{$\beta$} -- (0,0) --
(0,2.3) node[left]{$\alpha-\beta$};
 \draw [gray, very thin] (1,1) -- (1,0);
  \draw [gray,very thin] (1.5,1) -- (1.5,0);
   \draw [gray,very thin] (0.4,0) -- (0.4,0.4);
 \draw[dashed,red,  semithick, domain=0:1.5] plot (\x, {min(1,\x)});
   \draw  [semithick,blue] (.5,0) -- (1.5,1);
 \draw  [semithick,blue] (1.5,1) -- (3,1);
  \node at (1,2) {$n=4$};
 \node [left] at (0,1) { \tiny $1$ };
 \node [ below] at (0.57,0) { \tiny $\frac{1}{2}$ };
  \node [below] at (1,0) { \tiny $1$ };
   \node [below] at (1.5,0) { \tiny $\frac{3}{2}$ };
      \node [below] at (0.33,0) { \tiny $m$ };
\end{tikzpicture}
   \caption{The (red) dashed line represents the limitation on the regularity gain given in Theorem  \ref{teo:main1} for $q-q_B$, and the solid (blue)  line represents the positive results given in Theorem \ref{teo:main2}. When $n=2$,  the dot dashed line represents the previously known positive results of \cite{RRe}.}
   
   \label{fig.dim2y3}
\end{figure}
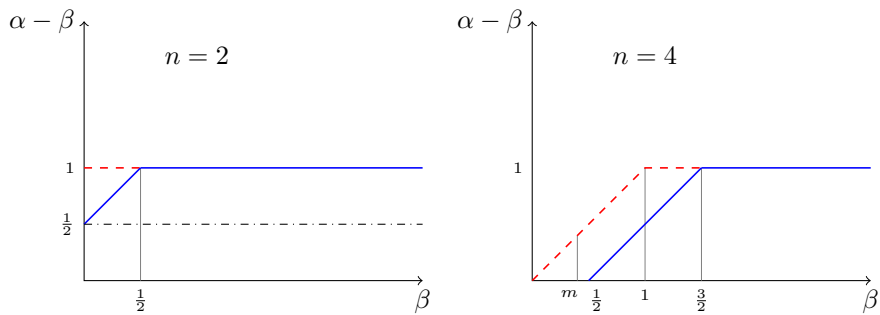

We introduce now the Sobolev estimates of the $Q_j$ operators. Consider a constant $C_0 \ge 1$ and  let $0 \le \chi(\xi) \le 1$, $\xi\in \mathbb{R}^n$, be a smooth cut-off function such that
 \begin{equation} \label{eq:cutoff1}
 {\chi}(\xi)= 1 \;\; \text{if} \;\; |\xi|>2C_0 \;\; \text{and} \;\; \chi(\xi)=0 \;\; \text{if} \;\; |\xi|<C_0.
 \end{equation}
We define the operator $\widetilde{Q}_{j}$ by the relation
%
\begin{equation}
\label{eq.cutoff}
\widehat{\widetilde{Q}_{j}(q)}(\xi ) := \chi (\xi )
\widehat{{Q}_{j}(q)}(\xi ),
\end{equation}
so that $Q_j(q)$ differs from $\widetilde Q_j(q)$ in a smooth function. $Q_j(q)$ will be introduced in the following section, see   (\ref{eq:Qjraw}).
 \begin{theorem}  \label{teo.Qj}
Let $n\ge 2$ and $j \ge 2$.  Assume that $0 \le \beta \le \infty$ and that the following condition also holds
 \begin{equation*}
 \alpha < \begin{cases}
 \,\, \beta + (j-1)(\beta - (n-3)/2), \quad if \quad  (n-3)/2 <\beta < (n-1)/2,\\
 \, \, \beta + (j-1), \hspace{27mm} if \quad (n-1)/2 \le \beta<\infty .
         \end{cases} 
\end{equation*}
Then for $q\in W_2^{\beta,2}(\RR^n)$ and $j=2$ we have the estimate 
\begin{equation} \label{eq:estimateQ2}
\norm{\widetilde {Q}_2(q)}_{{W}^{\alpha,2}}\le C  \norm{q}_{W^{\beta,2}_2}^2.
\end{equation}
Otherwise if $j\ge 3$ and $q\in W_4^{\beta,2}(\RR^n)$  we have that
\begin{equation} \label{eq:estimateQj}
\norm{\widetilde {Q}_j(q)}_{{W}^{\alpha,2}}\le C \norm{q}_{W_4^{\beta,2}}^j.
\end{equation}
\end{theorem}
To prove this result, we first obtain  an explicit formula  expressing the Fourier transform of ${\widetilde{Q}_{j}(q)}$ as certain  principal value distributions acting on the radial parameters of an integral operator over the Ewald spheres (the spherical operator), see    {Proposition    \ref{prop:Qjstruct}} below. In the proof of {Theorem    \ref{teo.Qj}} we use trace estimates to control the spherical
integrals, and a new method to reduce the estimate of the principal
value distributions to certain estimates of the spherical operators. One advantage of these techniques is that with the same effort we can prove estimates
for general dimension $n\ge 2$. In odd dimension, using very different
techniques, similar estimates for certain operators related to the $Q_{j}$  operators have been obtained in \cite[Theorem 1.1]{BM09quad} and
\cite[Theorem 1.2]{BM09}, for compactly
supported potentials. We mention that the estimate of the $
\widetilde{Q}_{3}$ operator for $n=3$ given in \cite{RRe} is still
the best estimate in the range $0 \le \beta <1/4$.

Finally, we give an upper bound for the regularity of the double dispersion operator, which constrains the amount  of  regularity  of $q$ that one  can expect to recover from the  Born approximation, as stated in     {Theorem    \ref{teo:main1}}.
\begin{theorem}  \label{teo:Q2count}
Let $0< \beta < \infty$ and  assume that $ Q_2(q) \in {W}_{loc}^{\alpha,2}(\RR^n)$ for every potential  ${q\in W^{\beta,2}(\RR^n)}$ radial, real  and compactly supported. Then $\alpha$ necessarily satisfies
\begin{equation} \label{eq.remarkable} 
\alpha \le  \begin{cases}
 \,\, 2\beta - (n-4)/2, \quad if \hspace{21mm} 0 \le \beta < (n-2)/2,\\
 \, \, \beta + 1, \hspace{18mm} if \hspace{6mm} (n-2)/2 \le \beta<\infty ,
 \end{cases} 
\end{equation}
\end{theorem}

The paper is structured as follows. In section    \ref{sec:2} we introduce
with more detail the backscattering problem, and we show how to deduce
    {Theorems    \ref{teo:main1} and \ref{teo:main2}} respectively from     {Theorems    \ref{teo:Q2count} and \ref{teo.Qj}}. Section    \ref{sec.PV} is dedicated to introducing the main result used for the estimate of the principal
value operators, and in section    \ref{sec.Q2L2} we estimate the spherical
part $\widetilde{Q}_{2}(q)$. In section    \ref{sec:Qj} we study the
general $\widetilde{Q}_{j}$ operators and we finish the proof of     {Theorem    \ref{teo.Qj}}. In section    \ref{sec:implicitQj} we give the estimates
necessary to show the convergence of the Born series in Sobolev spaces,
and section    \ref{sec:5} is devoted to proving     {Theorem    \ref{teo:Q2count}}.


\section{Convergence of the Born series in  Sobolev spaces} \label{sec:2}

Let us introduce the backscattering inverse problem more rigorously (see, for example, \cite[chapter V]{eskin} for an introduction to Schrödinger scattering theory). We follow a similar exposition to that of \cite{fix}.

Consider  a scattering solution $u_s(k,\theta,x)$, $k\in(0,\infty)$, $\theta\in \SP^{n-1}$,  of the stationary Schrödinger equation satisfying
\begin{equation} \label{eq.int.1}
\begin{cases}
(-\Delta + q -k^2)u=0 \\
u(x) = e^{i k \theta \cdot x} + u_s(k,\theta,x) \\
\lim_{|x| \to \infty} (\frac{\partial u_s}{\partial r} - iku_s)(x) = o(|x|^{-(n-1)/2}),
\end{cases}
\end{equation}
where the last line is the outgoing Sommerfeld radiation condition (necessary for uniqueness).  If $q$ is compactly supported, a solution $u_s$ of (\ref{eq.int.1}) has the following asymptotic behavior when $|x| \to \infty$
$$u_s(k,\theta,x) =  C |x|^{-(n-1)/2}k^{(n-3)/2} e^{ik|x|} u_\infty(k,\theta,x/|x|) + o(|x|^{-(n-1)/2}) ,$$
 for a certain function $u_\infty(k,\theta,\theta')$, $k\in (0,\infty)$, $\theta,\theta' \in \SP^{n-1}$. As mentioned in the introduction, $u_\infty$ is the so called scattering amplitude or far field pattern, and is given by the expression
\begin{equation} \label{eq.int.2}
u_\infty (k,\theta,\theta') = \int_{\RR^n} e^{-ik\theta'\cdot y} q(y) u(y) \, dy,
\end{equation}
where is important to notice that $u$ depends also on $k$ and $\theta$ (for a proof of this fact when $q\in C^\infty_c(\RR^n)$ see for example \cite{notasR}). 

Applying the outgoing resolvent of the Laplacian $R_k$ in the first line of (\ref{eq.int.1}), where 
\begin{equation} \label{eq:resolvent1}
\widehat{R_k(f)}(\xi)  = (-|\xi|^2+k^2 + i0)^{-1}\widehat{f}(\xi), 
\end{equation}
we obtain the Lippmann-Schwinger integral equation
\begin{equation} \label{eq.int.3}
u_s=R_k(qe^{ik\theta \cdot (\cdot)}) + R_k(qu_s(k,\theta,\cdot)).
\end{equation}
The existence and uniqueness of scattering solutions of (\ref{eq.int.1}) follows from a priori estimates for the resolvent operator $R_k$ and the previous integral equation (\ref{eq.int.3}). In the case of real potentials, this can be shown with  the help of Fredholm theory  for $k >0$, see for example \cite{notasR}. Otherwise, since the norm of the operator $T(f)= R_k(qf)$  decays to zero as $k \to \infty$ in appropriate function spaces, we can also use a Neumann series expansion in (\ref{eq.int.3}) which will be convergent for  $k >k_0$ (in general $k_0 \ge 0$  will depend on some a priori bound of $q$). For our purposes it is enough to consider $q \in L^r(\RR^n)$, $r>n/2$ and compactly supported. Notice that by the Sobolev embedding this is satisfied if $q\in W^{\beta,2}(\RR^n)$ with $\beta>(n-4)/2$.  See \cite[p. 511]{BFRV10} for more details and references.

We can introduce now the inverse backscattering problem. If we insert (\ref{eq.int.3}) in (\ref{eq.int.2}),  we can expand  the Lippmann-Schwinger equation in a Neumann series, as we mentioned before.  Then we obtain the Born series expansion relating  the scattering amplitude and the Fourier transform of the potential:
\begin{align}
\nonumber u_\infty (k,\theta,\theta') = \widehat{q}(\xi) + \sum^l_{j=2}\int_{\RR^n} e^{-ik \theta'\cdot y} (qR_k))^{j-1}(q(\cdot)e^{ik\theta \cdot (\cdot)} )(y) \,dy\\
 \label{eq.int.4} + \int_{\RR^n} e^{-ik \theta'\cdot y} (qR_k)^{l-1}(q(\cdot)u_s(k,\theta, \cdot) )(y) \,dy,
\end{align}
where $\xi = k(\theta'-\theta)$ and the last is the error term. Since we are considering complex potentials,  $u_\infty (k,\theta,\theta')$ is not defined for $k \le k_0$ as we have seen. Therefore we also have to ask  $k>k_0$ in (\ref{eq.int.4}). 

The problem of determining $q$ from the knowledge of the scattering amplitude is formally overdetermined in the sense that  the data $u_\infty(k,\theta,\theta')$ is described by $2n-1$ variables, while the unknown potential $q(x)$ has only $n$. We avoid the overdetermination by reducing to the backscattering data, assuming only  knowledge of $u_\infty(k,\theta,-\theta)$, for all $k > k_0$ and $\theta\in \SP^{n-1}$.  For backscattering data the problem is formally well determined, and the Born approximation $q_{B}$ is defined by the identity,
\begin{equation}  \label{eq:bornF}
 \widehat{q_{B}}(\xi):= u_\infty(k,\theta,-\theta), \hspace{4mm} \text{where} \hspace{3mm} \xi= -2k\theta.
 \end{equation}
Since $u_\infty(k,\theta,-\theta)$ is not defined for $k\le k_0$, from now on we consider  that $q_B(x)$ is defined modulo a $C^\infty$ function. 

By (\ref{eq:bornF}), the condition  $k>k_0$ is equivalent to asking $|\xi|>2k_0$. Therefore, using the cut-off introduced before (\ref{eq.cutoff}) with $C_0>2k_0$, and assuming convergence of the series, we can write  (\ref{eq.int.4})   as follows
\begin{equation} \label{eq.int.6}
 \chi(\xi)\widehat{q}_B(\xi) = \chi(\xi)\widehat{q}(\xi) + \sum_{j=2}^{\infty}\widehat{ \widetilde Q_{j}(q)}(\xi) ,
 \end{equation}
 where $\widetilde Q$ was defined in  (\ref{eq.cutoff}) and
\begin{equation} \label{eq:Qjraw} 
 \widehat{Q_{j}(q)}(\xi) =\int_{\RR^n} e^{ik \theta\cdot y} (qR_k)^{j-1}(q(\cdot)e^{ik\theta \cdot (\cdot)} )(y) \,dy,
 \end{equation}
again with  $\xi= -2k\theta$. 

We examine now the question of the convergence in Sobolev spaces of the series (\ref{eq.int.6}), an essential step in the proof of Theorems \ref{teo:main1} and \ref{teo:main2}.
\begin{proposition} \label{prop:convergence}
Let  $n\ge 2$, $j\ge 2$ and $\max(0,m) \le \beta<\infty$, where $m$ was defined in $(\ref{eq:m})$. If $q\in W^{\beta,2}(\RR^n)$  is compactly supported in $B_\rho$, the ball of radius $\rho$, then $\widetilde{Q}_j(q) \in W^{\alpha,2}(\RR^n)$ 
 if $\alpha<\alpha_j$, with
\begin{equation} \label{eq:alfaj}
\alpha_j = \beta + (j-1) - \frac{n}{2} - \frac{(n-1)}{2} (j-2) \max{\left( 0,\frac{1}{2} - \frac{\beta}{n} \right )}  .
\end{equation}
Moreover, for every $\alpha<\alpha_l$, $l\ge 2$ the series $\sum_{j=l}^{\infty} \widetilde Q_j(q),$
converges absolutely in $W^{\alpha,2}(\RR^n)$ provided we take ${C_0=C\norm{q}_{W^{\beta,2}}^{1/\varepsilon}}$ in $(\ref{eq:cutoff1})$ and $(\ref{eq.cutoff})$  for a large constant $C=C(n,\alpha,\beta,\rho)$ and a certain $\varepsilon =\varepsilon(n,\beta)>0$.
\end{proposition}

 This proposition  improves the original result of \cite[Proposition 4.3]{RV} given for the range  $m  \le \beta \le n/2$ and later extended in \cite{RRe} for $\beta\ge n/2$.  We have used certain properties of the fractional Laplacian $(-\Delta)^{s}$ to improve the value of  $\alpha_j$ in dimension $n$ (see section \ref{sec:implicitQj}).  It also improves the regularity gain given in \cite{RRe} for the $\widetilde Q_4$ operator with $n=3$. This would allow to obtain the results of recovery of singularities in that paper  without the very technical proof to estimate $\widetilde Q_4(q)$.
  
Using Proposition \ref{prop:convergence}, we can reduce the proof of Theorem \ref{teo:main2}  to proving  Theorem  \ref{teo.Qj}.

\begin{proof}[Proof of Theorem $\ref{teo:main2}$]
Taking the inverse Fourier transform of (\ref{eq.int.6}),  we can write {\it modulo a $C^\infty$ function}
\begin{equation}  \label{eq:bornseries}
q(x) -q_B(x) = - \sum_{j=2}^\infty \widetilde{Q}_j(x).
\end{equation}
Consider $\alpha$ and $\beta \ge 0$ satisfying condition (\ref{eq:mainrange}). Observe that by Theorem  \ref{teo.Qj}, we  have  $\widetilde Q_j(q) \in W^{\alpha,2}(\RR^n)$, so we just need to study the convergence of the series.  But, since the value of $\alpha_j$ in (\ref{eq:alfaj}) grows linearly with $j$,  we can always find an integer $l$ such that $\alpha_l > \alpha$. As a consequence, if $q$ has compact support, by Proposition \ref{prop:convergence}, the series $\sum_{j=l}^\infty \widetilde{Q}_j(q)(x)$ converges in $W^{\alpha,2}(\RR^n)$.  
\end{proof}

Similarly, the proof of Theorem \ref{teo:main1}  can be reduced to  Theorem  \ref{teo:Q2count} and Proposition \ref{prop:convergence}.

\begin{proof}[Proof of Theorem $\ref{teo:main1}$]

Take $\alpha \ge 0$ and assume that we have that $q-q_{B} \in W^{
\alpha ,2}_{loc}(\mathbb{R}^{n})$ for every compactly supported, real
and radial potential $q\in W^{\beta ,2}(\mathbb{R}^{n})$. We are going
to prove that then necessarily $Q_{j}(q) \in W^{\alpha ,2}_{loc}(
\mathbb{R}^{n})$ for all $2 \le j < \infty$.

We denote by ${ q}_{B}(\lambda )$ the Born approximation of the
potential $q(\lambda ) = \lambda q$, where $\lambda \in (0,1)$. By the
multilinearity of the $\widetilde{Q}_{j}$ operators, the Born series
{(\ref{eq:bornseries})} for $q(\lambda )$ becomes
\begin{equation}
\label{eq:algebra}
{\lambda q}- { q}_{B}(\lambda ) = - \sum _{j=2}^{l-1} \lambda ^{j}{\widetilde{Q}
_{j}(q)} -\sum _{j=l}^{\infty } \lambda ^{j}{\widetilde{Q}_{j}(q)},
\end{equation}
modulo a $C^{\infty }$ function (which depends on $\lambda $).

By Proposition \ref{prop:convergence} we have that if $m\le \beta <
\infty $, we can choose $l$ in (\ref{eq:algebra}) such that
$\alpha < \alpha _{l}$. Then $\widetilde{Q}_j(q) \in W^{\alpha,2}(\mathbb{R}^n)$ for $l \le j <\infty$, and the series $\sum _{j=l}^{\infty } \lambda ^{j}
{\widetilde{Q}_{j}(q)} $ will converge absolutely in $W^{\alpha ,2}(
\mathbb{R}^{n})$. 
Since by hypothesis we have that ${\lambda q}- { q}_{B}(\lambda )$ is
also a $W^{\alpha ,2}(\mathbb{R}^{n})$ function, by
(\ref{eq:algebra}) we have
\begin{equation*}
\sum _{j=2}^{l-1} \lambda ^{j} \widetilde{Q}_{j}(q)  \in W_{loc}^{\alpha,2}(\mathbb{R}^n).
\end{equation*}
But, by  choosing  $\lambda_i \in (0,1)$ for every $2 \le  i \le l-1$ such that $\det(\lambda_i^j) \neq 0$ (this is always possible, since it is a Vandermonde determinant), we obtain that, for all $2\le j \le l-1$,  $\widetilde{Q}_j(q) \in W^{\alpha,2}_{loc}(\mathbb{R}^n)$. 
To finish, notice that, by condition (\ref{eq.remarkable}) of
Theorem \ref{teo:Q2count}, we know that this implies that $\alpha $ must
be in the range given in (\ref{eq:mainrange1}).
\end{proof}

As we have mentioned in the introduction, the question of uniqueness of
the inverse scattering problem for backscattering data is still open.
In \cite{RU} it has been proved for $n=3$ that two potentials
differing in a finite number of spherical harmonics with radial
coefficients must be identical if they have the same backscattering
data. The question of uniqueness for small potentials was studied in
\cite{prosser4}. Generic uniqueness and uniqueness for small potentials
has been obtained in \cite{er3,st92} for dimensions 2 and 3
and in \cite{lagerg} for $n=3$. Similar results have been obtained
in odd dimension $n\ge 3$ in \cite{Uh01,BM09}, and in even dimension  in \cite{Wa}. See \cite{RU} for references about the uniqueness problem and \cite{er3} for more results concerning the regularity of the backscattering map.

The recovery of singularities has been studied in other inverse
scattering problems. The case of full data has been studied in
\cite{PSo,PSe,PSS} (real potentials) and \cite{BFRV10,fix}
(complex potentials) and the case of fixed angle data in
\cite{ser} in dimension $2$, and \cite{R} in dimension
$n\ge 2$. The regularity gain has been improved recently in
\cite{fix}. Analogous problems have been formulated  to
study the recovery of singularities of live loads in Navier elasticity, see
\cite{BFPRM1,BFPRM2}.

  Before going to the next section, we want to highlight the following property of Sobolev norms that we will use frequently in this work.
\begin{remark}
\label{remark:Sob}
We have that $W^{\beta ,2}_{\delta }\subset W^{\beta ',2}_{\delta '}$
if $\beta \ge \beta '$ and $\delta \ge \delta '$. This follows from the
equivalence
\begin{equation*}
\lVert <\cdot >^{\delta }<D>^{\beta } f\rVert _{L^{2}(\mathbb{R}^{n})}
\sim \lVert <D>^{\beta } <\cdot >^{\delta }f\rVert _{L^{2}(\mathbb{R}
^{n})},
\end{equation*}
and Plancherel theorem, see for example \cite[Definition 30.2.2]{HormanderIV}.
\end{remark}


\section{From the spherical integral to the P.V. integral} \label{sec.PV}

As we have explained in the introduction, the $Q_{j}$ operators that
appear in the Born series expansion of $q$ can be expressed in terms of certain spherical integrals, and principal value distributions acting on them. The usual strategy is to estimate the spherical part and then try to extend this estimate to the other terms. This is generally a very long and technical process that must be repeated case by case if the dimension or the value
of $j$ is changed (see \cite{RV,Re,RRe,BFRV13}). In this section we give a general method to reduce the estimate of the principal value distributions to the estimate of the spherical integrals

First, we define the following distributions. Let $f\in C^\infty_c((0,\infty))$, we put
\begin{equation}  \label{eq:dandP}
d(f) = \int_0^\infty \delta(1-r)f(r) \, dr, \hspace{3mm} \text{and} \hspace{3mm}P(f) = \pv \int_0^\infty \frac{1}{1-r} f(r) \, dr,
\end{equation}
where $\delta$ denotes  the Dirac delta distribution, as usual.
\begin{proposition} \label{prop:Q2struct}
 Let $r\in (0,\infty)$ and consider the modified Ewald spheres defined by the equation
\begin{equation} \label{eq:ewald}
\Gamma_r(\eta):=\{\xi\in\RR^n: |\xi-\eta/2| = r|\eta|/2\},
\end{equation}
(see Figure $\ref{fig.ewald1}$ in section $\ref{sec:5}$ below). Then we have that
\begin{equation} \label{eq:Q2} 
\widehat{Q_2(q)}(\eta) = (i \pi d + P) S_{r}(q)(\eta),
\end{equation}
where, if we denote by  $\sigma_{r\eta}$ the Lebesgue measure of $\Gamma_r(\eta)$,
\begin{equation} \label{eq:Sr}
{S_r(q)}(\eta) := \frac{2}{|\eta|(1+r)}  \int_{\Gamma_{r}(\eta)}\widehat{q}(\xi) \widehat{q}(\eta-\xi) \, \ds{r}(\xi).
\end{equation}
\end{proposition}
We omit the proof since it is just the case $j=2$ of Proposition \ref{prop:Qjstruct} below.

Motivated by the previous proposition, we introduce the following result to control the the principal value term. It will also simplify a
great amount of work when studying the $Q_{j}$ operators with
$j>2$.

\begin{proposition} \label{prop:genPV}
Let $1 \le p <\infty$ and $\alpha \in  \mathbb{R}$. Assume   that there is a $0<\delta<1$,  $\tau \in  \mathbb{R}$, $\gamma>0$ and $M>0$ such that the  one parameter family of $L^1_{loc}( \mathbb{R}^{n})$ functions $\{F_r\}_{r\in (0,\infty)}$  satisfies
\begin{enumerate}
\item For $a.e.$ $\eta \in \mathbb{R}^{n}$ fixed, $\partial_r F_r(\eta)$ is a continuous function for all $r\in (1-\delta,1+\delta)$, and in the same interval satisfies  the estimate
\begin{equation} \label{eq:pv:Fr1}
\lVert \partial_r F_r \rVert_{L_\tau^p} \le M .
\end{equation}
\item For every $r\in (0,\infty)$,
\begin{equation} \label{eq:pv:Fr2}
\lVert F_r\rVert_{L_\alpha^p} \le (1+r)^{-\gamma}M .
\end{equation}
\end{enumerate}
Then we have that
\begin{equation} \label{eq:pv:Fr3}
\lVert(i\pi d + P) F_r \rVert_{ L^p_{\alpha'}}\le  C_2 M,
\end{equation}
for every $\alpha'<\alpha$ and $C_2=C_2(\delta,\alpha,\alpha',\tau,p,\gamma)$.
\end{proposition}
Notice that the value of $\tau $ in  {(\ref{eq:pv:Fr1})} does not have any
influence on the value of $\alpha'$ in  {(\ref{eq:pv:Fr3})}.
\begin{proof} By the definition of $d$ in  {(\ref{eq:dandP})}, we clearly have that  $\lVert d \left( F_r \right) \rVert_{ L^p_{\alpha}}\le 2^{-\gamma}M$ follows directly  putting $r=1$ in  {(\ref{eq:pv:Fr2})}. Therefore it remains to estimate  the term 
\begin{equation*}  
{P(F_r)}(\eta) = \mathit{P.V.} \int_0^\infty \frac{1}{1-r} {F_r}(\eta) \,dr,
\end{equation*}
in $L^p_\alpha$. For some $s \ge 0$ that will be chosen later, set
\begin{equation} \label{eq:genPV.sing.1} 
\delta_\eta := {\delta}{ <\eta>^{-s}}.
\end{equation}
 Since for any $a>0$, $P.V. \int_{|1-r|<a} \frac{dr}{1-r} = 0$, we have that
\begin{align}
\nonumber P(F_r)(\eta) =& \\ 
\nonumber =& \int_{|1-r| \le   \delta_\eta} \frac {F_r - F_1}{1-r}(\eta) \,dr +\int_{\delta_\eta <|1-r|<\delta} \frac{F_r(\eta)}{1-r}dr + \int_{\delta \le |1-r|} \frac{F_r(\eta)}{1-r}\,dr \\
  \label{eq:noPV}  :=& \, P_{A}(\eta) + P_{B}(\eta) + P_{C}(\eta),
\end{align}
 where  the $\mathit{P.V.}$ is not necessary any more, since we can cancel the singularity in the denominator thanks to the fact that $F_r(\eta)$ is $C^1$ in $(1-\delta,1+\delta)$ by the first condition. Applying Minkowski's integral inequality and estimate (\ref{eq:pv:Fr2}), we obtain that
\begin{equation}\label{eq:PV1:B}
\|P_{C} \|_{L^p_\alpha} \le \int_{ \delta<|1-r|} \frac{\|F_r\|_{L^p_\alpha} }{|1-r|} \, dr \le C(\delta,\gamma) M.
\end{equation} 
By the fundamental theorem of calculus we have
\[ \frac{F_r(\eta)-F_1(\eta)}{1-r} = - \int^{1}_{0} \partial_u F_{u} (\eta) \big |_{u = u(t) } \, dt ,\]
where $u(t) = (r-1)t +1 $.  Then the inequality  
\[ 
<\eta>^{s/2} \, \le \delta^{1/2}|1-r|^{-1/2}, 
\] 
that holds in the region $|1-r|<\delta_\eta$, yields
\begin{align*}
&\lVert P_A \rVert_{ L^p_{\alpha}} = \left( \int_{\mathbb{R}^n} <\eta>^{p\alpha}
\left |\int_{|1-r| <  \delta_\eta}  \int^{1}_{0} \partial_u F_{u}(\eta)\big |_{u=u(t)} \, dt \,dr\right|^p d\eta \right)^{1/p} \\
 &\le\delta^{1/2} \left( \int_{\mathbb{R}^n} <\eta>^{p(\alpha -s/2)}
\left (\int_{|1-r| <  \delta}  |1-r|^{-1/2} \int^{1}_{0} \left | \partial_u F_{u}(\eta)\big |_{u=u(t)} \right | \, dt \,dr\right)^p d\eta \right)^{1/p} \\
 &\le \delta^{1/2} \int_{|1-r| <  \delta}  \int^{1}_{0} |1-r|^{-1/2} \lVert \partial_u F_{u}\big |_{u=u(t)}\rVert_{L^p_{\alpha-s/2}} \, dt \,dr ,
\end{align*}
where to get the last line we have used  Minkowski's inequality.
 
 We have two cases. If in (\ref{eq:pv:Fr2}) and (\ref{eq:pv:Fr1}) we have $\alpha \le \tau$ we can choose $s = 0$, otherwise, if $\alpha>\tau$, we choose $s$ such that $\alpha-s/2 = \tau$.
In both cases by (\ref{eq:pv:Fr1}) we obtain 
\begin{align} 
\nonumber \lVert  P_A \rVert_{ L^p_\alpha} &\le  \delta^{1/2} \int_{|1-r| <  \delta} \int^{1}_{0}  |1-r|^{-1/2} \lVert  \partial_u F_{u}\big|_{u=u(t)} \rVert_{L^p_{\tau}} \, dt \,dr  \\
 \label{eq:PV2:B} &\le 4\delta M.
\end{align}

To finish we need estimate $P_B$  which is non-zero when $s>0$, that is when $\alpha>\tau$.
We set $N(\eta) = - \log_2(\delta <\eta>^{-s})$, and consider the next dyadic decomposition,
\begin{align*}
P_{B}(\eta) :&=  \int_{B } \frac{F_r(\eta)}{1-r}  \,dr \\
&= \sum_{0 \le j<N(\eta)}  \int_{\{2^{-(j+1)}<|1-r|<2^{-j}\}} \chi_{\{|1-r|<\delta_\eta \}}(r)  \frac{F_r(\eta)}{1-r}  \,dr .
\end{align*}
 If $j=0,1,...,N(\eta)$, for $\eta$ fixed, the definition of $N(\eta)$ implies that ${2^j \le  <\eta>^{s}}/\delta$, therefore
 \begin{equation}  \label{eq:diadic1:B}
|P_{B}(\eta)|\le \sum_{j=0}^\infty 2^{j+1} \chi_{(\delta 2^{j},\infty)}(<\eta>^{s})\int_{|1-r|<2^{-j}} |{F_r}(\eta)| \,dr .
\end{equation}
 But observe that in the last line we have an expression of the kind
\[{P^{\lambda}}(\eta) := \chi_{(\delta{\lambda^{-1}},\infty)} (<\eta>^{s}) \int_{|1-r|\le \lambda} |{F_r}(\eta)| \, dr,\]
 with $0<\lambda\le 1$. Computing its $L^p_{\alpha-\varepsilon}$ norm when $\varepsilon>0$ and applying Minkowski's integral inequality we obtain
\begin{equation}  \label{eq:diadic2:B}
\lVert  P^{\lambda}\rVert_{ L^p_{\alpha-\varepsilon}}  \le {\lambda}^{\varepsilon/s}  \int_{\{|1-r|\le {\lambda} \}} \lVert  F_r \rVert_{ L^p_\alpha} \,dr \le {\lambda}^{1+\varepsilon/s}  M,
\end{equation}
where we have used  estimate (\ref{eq:pv:Fr2}), and that in the region where the characteristic function does not vanish we have that $<\eta>^{-\varepsilon} \le \delta^{-\varepsilon/s} \lambda^{\varepsilon/s}$. Hence, taking the $L^p_{\alpha'}$ norm of (\ref{eq:diadic1:B}) and applying estimate (\ref{eq:diadic2:B}) with $\varepsilon= \alpha-\alpha'$ yields
\begin{align} 
\nonumber \lVert  P_B \rVert_{ L^p_{\alpha'}} \le 2 \sum^{\infty}_{j=0}  2^{j} \lVert   {P^{2^{-j}}} \rVert_{ L^p_{\alpha'}} 
&\le    2\delta^{-\varepsilon/s} M \sum^{\infty}_{j=0}  2^{-j \varepsilon/s}\\
\label{eq:PV3:B}&\le   C(\delta,\alpha,\alpha',\tau,p) M.
\end{align}
 Observe that this is the first time we need the strict inequality $\alpha'<\alpha$ in the statement of the theorem. Therefore since $P(F_r) = P_A+P_B+P_C$ we conclude the proof putting together estimates (\ref{eq:PV1:B}), (\ref{eq:PV2:B}) and (\ref{eq:PV3:B}). 
\end{proof}
In our case usually the family of functions $F_r$ is given by a multilinear  operator over the potentials, as it can be seen in  \eqref{eq:Q2} where $F_r= S_r(q)$.


\section{Sobolev  estimates for the double dispersion operator} \label{sec.Q2L2}

In this section, we study in detail the spherical operator $S_{r}$ of
the double dispersion operator $\widetilde{Q}_{2}$ in order to prove
{Theorem \ref{teo.Qj}} for $j=2$. This section will serve to illustrate
the approach used to obtain in Section \ref{sec:Qj} the main estimates
of the spherical operators related to the $\widetilde{Q}_{j}$ operators.

For notational convenience we define the operator
\begin{equation*}
{\widetilde{S}_{r}(q)}(\eta ) :=\chi (\eta ) {S_{r}(q)}(\eta ).
\end{equation*}
Then, multiplying both sides of equation {(\ref{eq:Q2})} by the smooth
cut-off $\chi (\eta )$ we get
\begin{equation}
\label{eq:Q2tilde}
\widehat{\widetilde{Q}_{2}(q)}(\eta ) = (i \pi d + P) \widetilde{S}
_{r}(q)(\eta ).
\end{equation}
Hence, the main idea to estimate the $\widetilde{Q}_{2}$ operator is to
apply {Proposition \ref{prop:genPV}} to the particular case $F_{r}= \widetilde{S}_{r}(q)$. We begin with the necessary estimates for
$ \widetilde{S}_{r}(q)$.
  \begin{lemma} \label{lemma:SrQ2}
Let $n\ge 2$ and $q\in W_1^{\beta,2}(\RR^n)$ with $\beta\ge 0$. Then the estimate
\begin{equation*}
\norm{\widetilde S_r(q)}_{ L^2_\alpha}\le C{(1+r)^{-\gamma} } \norm{q}_{W_1^{\beta,2}}^2, 
\end{equation*}
holds  when
\begin{equation} \label{eq:range:S2}
\begin{cases}
 \alpha \le \beta +(\beta- (n-3)/2), \quad if \quad (n-3)/2 <\beta < (n-1)/2,\\
 \alpha<\beta +1, \hspace{26mm} if \quad (n-1)/2 \le \beta<\infty ,
         \end{cases} 
\end{equation}
for some real number $\gamma >0$ (possibly depending on $\beta$ and $\alpha$).
\end{lemma}
To simplify later computations we define the operator  
$$ {\widetilde{K}_{r}(g_1,g_2)(\eta)} = \chi(\eta){K}_{r}(g_1,g_2)(\eta),$$
where
\begin{equation}  \label{eq:K2}
{{K}_{r}(g_1,g_2)}(\eta) := \frac{1}{|\eta|}\int_{\Gamma_{r}(\eta)}|g_1(\xi)| |g_2(\eta-\xi)| \, \ds{r}(\xi).
\end{equation}
Then we have that
$$
\left| \widetilde{S}_r(q)(\eta) \right| \le  \frac{2}{1+r}{\widetilde{K}_{r}(\widehat{q},\widehat{q})(\eta)},
$$
and therefore the proof of Lemma \ref{lemma:SrQ2} is an immediate consequence of the following lemma taking $\gamma =1- \lambda$.
\begin{lemma} \label{lemma.K.Q2}
Let $n\ge 2$ and $f_1,f_2 \in W_1^{\beta,2}(\RR^n)$ with $\beta\ge 0$. Then the estimate
\begin{equation} \label{eq.thm.sph.1}
\norm{\widetilde K_{r}(\widehat{f_1},\widehat{f_2})}_{ L^2_\alpha}\le C r^{\lambda} \norm{f_1}_{W_1^{\beta,2}} \norm{f_2}_{W_1^{\beta,2}},
\end{equation}
 holds when  condition $(\ref{eq:range:S2})$ is also satisfied, for some real number  $0<\lambda<1$ (possibly depending on $\beta$ and $\alpha$).
\end{lemma}
In the proof we are going to use the following result.
\begin{lemma}[\cite{fix}, Lemma 3.3]\label{lemma.integrals}
Let $\SP_\rho\subset\RR^n$ be any sphere of radius $\rho$ and let $d\sigma_\rho$ be its Lebesgue measure.
Then for any $0< \lambda\le (n-1)/2$, we have that
\begin{equation*} 
\int_{\SP_\rho} \frac{1}{|x-y|^{(n-1)-2\lambda}} \, d\sigma_\rho(y) \le C_\lambda \rho^{2\lambda},
\end{equation*}
 for any $x\in\RR^n$, and for a constant $C_\lambda$ that only depends on  $\lambda$. 
\end{lemma}
This can be proved by direct computation (for a detailed proof see \cite[Appendix]{fix}).
\begin{proof}[Proof of Lemma $\ref{lemma.K.Q2}$]
Since by \eqref{eq:cutoff1}, $\chi(\eta) = 0$ for $|\eta|\le 1$, we have that $|\eta|^{-1} \le  2 \jp{\eta}^{-1}$ in the region where $\chi(\eta)$ does not vanish. Then
\begin{equation*}
\norm{\widetilde{K}_r({\widehat{f_1}}, \widehat{f_2})}_{L^2_\alpha}^2 \le C \int_{\RR^n} \jp{\eta}^{2\alpha-2} \left( \int_{\Gamma_{r}(\eta)} |\widehat{f_1}(\xi)| |\widehat{f_2}(\eta-\xi)| \,\ds{r}(\xi) \right)^2  d\eta.
\end{equation*}
Now, $\eta = (\eta-\xi) + \xi$, so if we choose any $0<c<1/2$ at least one of the conditions $|\xi|>c|\eta|$ and $|\eta-\xi|>c|\eta|$ must hold.
But observe now that the change of variables $\xi' = \eta -\xi$ leaves invariant $\Gamma_r(\eta)$ and $\widetilde{K}_r({\widehat{f_1}}, \widehat{f_2})$, except for the fact that interchanges the roles of $\widehat{f_1}$ and $\widehat{f_2}$. Therefore is enough to study only the case of $|\xi|>c|\eta|$ since then the other follows applying the change of variables. We want to estimate
\begin{equation*} 
I :=\int_{\RR^n} \jp{\eta}^{2\alpha-2}  \left( \int_{\Gamma^+_{r}(\eta)} |\widehat{f_1}(\xi)| |\widehat{f_2}(\eta-\xi)| \,\ds{r}(\xi) \right)^2 d\eta,
\end{equation*}
$$ \text{where} \hspace{3mm} \Gamma_r^+(\eta) := \{\xi\in \Gamma_r(\eta):|\xi|>c|\eta|\}.$$
 We introduce a real parameter $0<\lambda \le (n-1)/2$. Then by Cauchy-Schwarz's inequality we have
\begin{align}
\nonumber &I \le C  \int_{\RR^n} \jp{\eta}^{2\alpha-2} \int_{\Gamma_{r}^+(\eta)} |\widehat{f_1}(\xi)|^2|\widehat{f_2}(\eta-\xi)|^2|\eta-\xi|^{n-1-2\lambda}\,\ds{r}(\xi) \times \dots\\
\nonumber &\hspace{60 mm} \dots \times \int_{\Gamma_{r}^+(\eta)}\frac{1}{|\eta-\xi|^{n-1-2\lambda}}\,\ds{r}(\xi)\, d\eta .
\end{align} 
Since $\Gamma_r(\eta)$ has radius $r|\eta|/2$, using Lemma \ref{lemma.integrals} to bound the second integral we obtain
\begin{align}
\nonumber &I \le C r^{2\lambda}\int_{\RR^n} {\jp{\eta}^{2\alpha-2}}|\eta|^{2\lambda}\times \dots \\ 
& \nonumber \hspace{2cm}\dots \times  \int_{\Gamma_{r}^+(\eta)} |\widehat{f_1}(\xi)|^2|\widehat{f_2}(\eta-\xi)|^2 \jp{\eta-\xi}^{n-1-2\lambda}\,\ds{r}(\xi)\, d\eta \\
\label{eq.Q2sph.1} &\le C r^{2\lambda}\int_{\RR^n} \int_{\Gamma_{r}(\eta)} |\widehat{f_1}(\xi)|^2{\jp{\xi}^{2\alpha-2+2\lambda}}|\widehat{f_2}(\eta-\xi)|^2 \jp{\eta-\xi}^{n-1-2\lambda}\,\ds{r}(\xi)\, d\eta,
\end{align}
using also that $\jp{\eta}^{2\alpha-2+2\lambda} \le C{\jp{\xi}^{2\alpha-2+2\lambda}}$, which follows from the fact that $|\eta| \le c |\xi|$, if we impose the extra condition ${\alpha-1+\lambda}\ge 0$.

We are going to use the trace theorem to bound  second integral in {(\ref{eq.Q2sph.1})}. The fundamental point is that for spheres, the constant of the trace theorem can be taken to be $1$,  independently of the radius of the sphere. See \cite[Proposition A.1]{fix},  for an elementary proof of this fact. Then
\begin{align}
\nonumber &I \le   \, Cr^{2\lambda} \int_{\RR^n} \int_{\RR^n} | \widehat{f_1}(\xi)|^2\jp{\xi}^{2\alpha-2+2\lambda}|\widehat{f_2}(\eta-\xi)|^2 \jp{\eta-\xi}^{(n-1)-2\lambda}  d\xi\,d\eta \\
\nonumber & \, + Cr^{2\lambda}\int_{\RR^n} \int_{\RR^n} \left |\nabla \left( \widehat{f_1}(\xi)\jp{\xi}^{\alpha-1+\lambda}\right) \right |^2  |\widehat{f_2}(\eta-\xi)|^2 \jp{\eta-\xi}^{(n-1)-2\lambda} d\xi \,d\eta\\
\nonumber &\,  + \, Cr^{2\lambda} \int_{\RR^n} \int_{\RR^n} | \widehat{f_1}(\xi)|^2\jp{\xi}^{2\alpha-2+2\lambda}\left |\nabla \left(\widehat{f_2}(\eta-\xi) \jp{\eta-\xi}^{(n-1)/2-\lambda}\right)\right|^2 d\xi\,d\eta.
\end{align} 
Therefore changing the order of integration and using that by  Plancherel theorem we have
\begin{equation*}
\int_{\RR^n} \left|\nabla (\widehat{f}(\xi)\jp{\xi}^{t} )\right|^2   \, d\xi \le C \norm{f}_{W^{t,2}_1}^2,
\end{equation*}
we obtain
$$I \le Cr^{2\lambda} \norm{f_1}_{W^{\alpha -1+\lambda,2}_1}^2\norm{f_2}_{W^{(n-1)/2-\lambda,2}_1}^2.$$
 As we have explained before, in the case $|\eta-\xi|>c|\eta|$ we obtain the same estimate  but interchanging the roles of $f_1$ and $f_2$. Putting both estimates together we get
\begin{align*}
&\norm{\widetilde{K}_r({\widehat{f_1},\widehat{f_2}})}_{L^2_\alpha} \\
&\le C r^{\lambda} \left( \norm{f_1}_{W^{\alpha -1+\lambda,2}_1}\norm{f_2}_{W^{(n-1)/2-\lambda,2}_1} + \norm{f_2}_{W^{\alpha -1+\lambda,2}_1}\norm{f_1}_{W^{(n-1)/2-\lambda,2}_1} \right).
\end{align*}
We also add the extra restriction $\lambda<1$, this is necessary to have  a negative value for $\gamma$ in Lemma \ref{lemma:SrQ2}. Now, fix  $\lambda$ such that
\begin{equation} \label{eq.par.Q2}
\beta =  \alpha -1+\lambda,
\end{equation}
 hence, the condition $\alpha -1+\lambda \ge 0$ used in the proof implies we must have $\beta \ge 0$. As a consequence of (\ref{eq.par.Q2}), equation (\ref{eq.thm.sph.1}) follows directly in the range $\beta \ge (n-1)/2$ (we are using remark \ref{remark:Sob}). But, by the conditions imposed in the proof we have to take into account the restrictions
\begin{equation}\label{restriccion_1}
\begin{cases}
0 < \lambda < 1 \\ 0 < \lambda \le \frac{n-1}{2}
\end{cases}   \Longleftrightarrow \begin{cases}
\beta < \alpha < \beta +1  \\ \beta +1 - \frac{n-1}{2} \le \alpha < \beta +1.
\end{cases}                
\end{equation}
We can discard the lower bounds for  $\alpha$ using that $\norm{f}_{L^2_\alpha} \le \norm{f}_{L^2_{\alpha'}}$  always holds if $\alpha\le \alpha'$. Therefore   only the restriction $ \alpha<\beta +1$ remains.

Otherwise, if $\beta$ is in the range $0 \le \beta < (n-1)/2 $, estimate (\ref{eq.thm.sph.1}) will follow if we add the extra condition  
\begin{equation} \label{rest2}
(n-1)/2-\lambda \le \beta.
\end{equation}
Then, since $\lambda<1$, we must have $\beta>(n-3)/2 $ (the other conditions on $\lambda$ don't add new restrictions). Also (\ref{eq.par.Q2}) and (\ref{rest2}) imply together  that $\alpha \le 2\beta-(n-3)/2$, which is a stronger condition than $\alpha<\beta+1$ since we have $\beta<(n-1)/2$. Hence, we have obtained the ranges of the parameters given in the statement.
\end{proof}
\begin{lemma} \label{lemma:Kpoint}
Let $q \in \mathcal S'(\mathbb{R}^n)$ such that $\widehat{q}$ is smooth. Then, for every $\eta \neq 0$ fixed, $S_r(q)(\eta)$ is smooth in the $r$ variable. Moreover, we have the following pointwise inequality
\begin{equation} \label{eq:Kpoint}
  \left | \partial_r  S_{r}(q)(\eta) \right | \le C  K_{r}(\widehat{q},\widehat{q})(\eta) + C |\eta|\sum_{i=1}^n K_{r}(\widehat{x_iq},\widehat{q}).
 \end{equation}
\end{lemma}
In general, the constant $C$ in the estimate might depend on $\delta$, but this is harmless.
\begin{proof}
 We centre the Ewald sphere in (\ref{eq:Sr}) at the origin with the change $\xi = \eta/2 + r|\eta/2|\theta$, where $\theta \in \SP^{n-1}$, to obtain
\begin{equation} \label{eq:derivative1}
 S_{r}(q)(\eta) =\frac{r^{n-1}|\eta|^{n-2}}{2^{n-2}(1+r)} \int_{\SP^{n-1}}\widehat{q}\left(r\frac{|\eta|}{2}\theta + \frac{\eta}{2} \right) \widehat{q}\left(-r\frac{|\eta|}{2}\theta + \frac{\eta}{2}\right) \, d\sigma(\theta).
\end{equation} 
Now we can compute derivatives in the $r$ variable. Consider $\eta$ fixed, then
\begin{align*}
\nonumber & \partial_r S_{r}(q)(\eta) = \\
 \nonumber &=  \frac{((n-1)r^{n-2}(1+r)-r^{n-1})|\eta|^{n-2}}{2^{n-2}(1+r)^2} \int_{\SP^{n-1}}\widehat{q}\left(r\frac{|\eta|}{2}\theta + \frac{\eta}{2} \right) \widehat{q}\left(-r\frac{|\eta|}{2}\theta + \frac{\eta}{2}\right) \, d\sigma(\theta) \\
\nonumber &+  \frac{r^{n-1}|\eta|^{n-1}}{2^{n-1}(1+r)} \int_{\SP^{n-1}}\theta \cdot\nabla \widehat{q}\left(r\frac{|\eta|}{2}\theta + \frac{\eta}{2} \right) \widehat{q}\left(-r\frac{|\eta|}{2}\theta + \frac{\eta}{2}\right) \, d\sigma(\theta) \\
  &-  \frac{r^{n-1}|\eta|^{n-1}}{2^{n-1}(1+r)} \int_{\SP^{n-1}} \widehat{q}\left(r\frac{|\eta|}{2}\theta + \frac{\eta}{2} \right)\theta \cdot \nabla \widehat{q}\left(-r\frac{|\eta|}{2}\theta + \frac{\eta}{2}\right) \, d\sigma(\theta). 
\end{align*}
We have passed the derivative inside the integral since we are
integrating in finite measure and $\widehat{q}$ is smooth. This  implies
that $ S_{r}(q)(\eta )$ is smooth in the $r$ variable 
for every $\eta \neq 0$.  Also, a change of variables $\omega = -\theta $ shows that the
last terms are identical. Hence, if we undo the change to spherical
coordinates we get
\begin{align} 
 \nonumber \partial_r S_{r}(q)(\eta)=&   \frac{(n-2)r+(n-1)}{r(1+r)^2}\frac{2}{|\eta|} \int_{\Gamma_r(\eta)}\widehat{q}(\xi) \widehat{q}(\eta-\xi) \, \ds{r}(\xi)  \\
 \label{eq:derivative2} +& \frac{2}{(1+r)} \int_{\Gamma_r(\eta)} \frac{(\xi-\eta/2)}{|\xi-\eta/2|} \cdot\nabla \widehat{q}(\xi) \widehat{q}(\eta-\xi) \, \ds{r}(\xi).
\end{align}
Therefore by (\ref{eq:K2}), if we fix some $0<\delta<1$,  for $r\in (1-\delta,1+\delta)$ we obtain
\begin{equation*} 
  \left | \partial_r  S_{r}(q)(\eta) \right | \le C  K_{r}(\widehat{q},\widehat{q})(\eta) + C |\eta|  K_{r}(|\nabla\widehat{q}|,\widehat{q})(\eta).
 \end{equation*}
 The estimate follows then using that
\begin{equation*} 
 K_{r}(|\nabla\widehat{q}|,\widehat{q})\le  \sum_{i=1}^n K_{r}(\partial_i \widehat{q},\widehat{q})= C \sum_{i=1}^n K_{r}(\widehat{x_iq},\widehat{q}).
 \end{equation*}
\end{proof}
From Lemmas \ref{lemma:SrQ2} and \ref{lemma:Kpoint}  we  get the following proposition.
\begin{proposition}
\label{prop:derSr}
Let $n\ge 2$ and fix some $0<\delta <1$. Then for every $r\in (1-
\delta ,1+\delta )$ and $q \in \mathcal{S}(\mathbb{R}^{n})$ we have
that
\begin{equation*}
\lVert \partial _{r} \widetilde{S}_{r}(q)\rVert _{ L^{2}_{\alpha -1}}
\le C \lVert q\rVert _{W^{\beta ,2}_{2}}^{2},
\end{equation*}
holds when $\alpha $ and $\beta \ge 0$ satisfy condition
${(\ref{eq:range:S2})}$. 
\end{proposition}
 Notice the appearance of the Sobolev space  $W^{\beta,2}_2$ instead of $W^{\beta,2}_1$.
 \begin{proof}
Multiplying (\ref{eq:Kpoint}) by $\chi(\eta)$ we get
\begin{equation*} 
  \norm{\partial_r \widetilde S_{r}(q)}_{L^2_{\alpha-1}} 
   \le C  \norm{ \widetilde{K}_{r}(\widehat{q},\widehat{q}) }_{L_{\alpha-1}^2}+  C \sum_{i=1}^{n} \norm{\widetilde K_{r}(\widehat{x_iq},\widehat q)}_{ L^2_\alpha}.
\end{equation*}
Notice that we get the $L^2_{\alpha}$ norm in the last term  due to the extra $|\eta|$ factor  appearing in (\ref{eq:Kpoint}). Then, by  Lemma \ref{lemma.K.Q2} we obtain the desired estimate using that
\begin{equation} \label{eq:pain}
 \norm{\widetilde K_{r}(\widehat{x_iq},\widehat q)}_{ L^2_\alpha}\le C  \norm{x_iq}_{W_1^{\beta,2}} \norm{q}_{W_1^{\beta,2}} \le C  \norm{q}_{W_2^{\beta,2}}^2.   
\end{equation} 
The estimate $\norm{x_iq}_{W_1^{\beta,2}} \le C  \norm{q}_{W_2^{\beta,2}}$ can be verified  for integer $\beta$ and extended by interpolation to the general case.

 \end{proof}

By Lemma \ref{lemma:SrQ2} and \ref{prop:derSr} we can apply Proposition \ref{prop:genPV} to estimate the $\widetilde{Q}_2$ operator, but we leave this for the next section.



\section{Sobolev estimates for the general \texorpdfstring{$\widetilde{Q}_j$}{Qj} operator } \label{sec:Qj}

In this section we prove Theorem \ref{teo.Qj}.

 Let $\ell \ge 1$, and assume we have $\mb r \in (0,\infty)^{\ell}$,  $\mb r  = (r_1, \dots,r_{\ell})$ and  $f\in C^{\infty}_c((0,\infty)^{\ell})$. We define the  operators, 
$$P_i,d_i : C^{\infty}_c((0,\infty)^{\ell}) \to C^{\infty}_c((0,\infty)^{\ell-1}),$$
 following the notation introduced in (\ref{eq:dandP}),
\begin{align*}
d_i(f)(r_1,\dots,\widehat{r_i},\dots,r_\ell) &:= \int_0^\infty \delta(r_i-1)f(\mb r) \, dr_i, \\
 P_i(f)(r_1,\dots,\widehat{r_i},\dots,r_\ell) &:= \pv \int_{0}^\infty \frac{1}{1-r_i} f(\mb r) \, dr_i,
\end{align*}
where  $\widehat{r_i}$ indicates that this coordinate is deleted in the list. Hence, if $\ell=1$, $d_i(f)$ and $P_i(f)$ are just scalars. Also, if  $\mb r \in (0,\infty)^\ell$ we define the manifold,
$$\Gamma_{\mb r}(\eta) =  \Gamma_{ r_1}(\eta) \times  \dots  \times  \Gamma_{ r_\ell}(\eta) ,$$
and we denote by $\sigma_{\mb r} $ its Lebesgue measure (product of the measures of the spheres $\Gamma_{r_i}(\eta)$),
$$d \sigma_{\mb r}(\xi_1,\dots,\xi_\ell) = d\sigma_{r_1\eta}(\xi_1) \times\dots\times \, d\sigma_{r_\ell\eta}(\xi_\ell).$$
\begin{proposition}[$Q_j(q)$ structure] \label{prop:Qjstruct}
 Let $n\ge 2$ and $j\ge 2$. Then we have that
 \begin{equation} \label{eq:Qj}
 \widehat{Q_j(q)}(\eta) =   \prod_{i=1}^{j-1} \left( i\pi d_i + P_i \right)  S_{j,\mb r} (q)(\eta),
 \end{equation}
 where $\mb r  = (r_1, \dots,r_{j-1})$, and
 \begin{align}
  \nonumber  S_{j,\mb r}(q)&(\eta) := \left(\prod_{i=1}^{j-1} \frac{2}{1+r_i}\right) \times \dots
  \\ \label{eq:Sj} &\frac{1}{|\eta|^{j-1}} \int_{\Gamma_{\mb r}(\eta)}   \widehat{q}(\eta -\xi_1) \left(\prod_{i=1}^{j-2} \widehat{q}(\xi_i - \xi_{i+1}) \right) \widehat{q}(\xi_{j-1}) \,\, d\sigma_{\mb r}(\xi_1,\dots,\xi_{j-1}).
 \end{align}
\end{proposition}
Proposition \ref{prop:Qjstruct} implies that the higher order operators $Q_j$ have a similar structure to the $Q_2$ operator (we consider $\prod_k^m$ =1 if $k>m$, as it is usual). In fact, when $j=2$, (\ref{eq:Qj}) is  equivalent to equation (\ref{eq:Q2}) since with the new notation we have  $ S_r =S_{2,\mb r}$ (in this case we  have $\mb r = r$, since there is only one parameter).  
\begin{proof}
 Let $k\in(0,\infty)$, we are going need the identity
\begin{equation} \label{eq:resAlberto}
R_k(f)(x) = i \frac{\pi}{2}k^{n-2} \int_{\SP^{n-1}} \widehat{f}(k\omega) e^{ik x \cdot \omega} \, d\sigma(\omega) + \pv \int_{\RR^n} e^{i x \cdot \zeta} \frac{\widehat{f}(\zeta)}{-|\zeta|^2 + k^2} \, d\zeta ,
\end{equation}
 for the resolvent of the Laplacian. (It follows from computing explicitly the limit in (\ref{eq:resolvent1}) in the sense of distributions, see for example  \cite{notasR} and \cite[pp. 209-236]{GS} for more details).
 We take spherical coordinates in the principal value integral, denoting by  $t$ the radial variable and use the change of variables $t=rk$ in the radial integral, 
 \begin{align*} 
  \pv \int_{\RR^n} e^{i x \cdot \zeta} \frac{\widehat{f}(\zeta)}{-|\zeta|^2 + k^2} \, d\zeta &=    \pv\int_0^\infty \frac{1}{(k-t)(k+t)} \int_{\mathbb{S}^{n-1}}e^{ix \cdot t\omega}\widehat{f}(t\omega) \, t^{n-1}d \sigma(\omega)\, dt \\
 = \pv \, \frac{1}{k}\int_0^\infty &\frac{1}{(1-r)(1+r)}\int_{\mathbb{S}^{n-1}}e^{ix \cdot rk\omega} \widehat{f}(rk \omega) \, (rk)^{n-1} d \sigma(\omega)\, dr \\
 = \pv \frac{1}{k}\int_0^\infty &\frac{1}{(1-r)(1+r)}\int_{\Gamma_r(\eta)}e^{i(-\xi-k \theta) \cdot x}\widehat{f}(-\xi -k \theta) \,d \sigma_{r\eta}(\xi)\, dr,
 \end{align*} 
where to obtain the integral over the Ewald sphere in the last line we have used the change of variables   $rk\omega= -\xi-k\theta$ in the spherical integral, and that $\Gamma_r(\eta) = \left\{ \xi \in \mathbb{R}^n: \; |\xi+k\theta|=rk  \right\}$ if $\eta= -2k\theta$ (see (\ref{eq:ewald})).
Hence,  using the analogous change of variables $k \omega=-\xi-k \theta$ in the first integral in (\ref{eq:resAlberto}),  we finally obtain
\begin{align}
 \nonumber R_k(f)(x)&= i \pi \frac{1}{|\eta|}\int_{\Gamma_1(\eta)} e^{i  (-\xi-k\theta) \cdot x} \widehat{f}(-\xi -k\theta) \, d\sigma_\eta(\xi) \\
   \nonumber &\hspace{20mm} + \pv \int_0^\infty \frac{2}{|\eta|(1+r)} \int_{ \Gamma_{r}(\eta)} e^{i  (-\xi-k\theta) \cdot x}  \, \widehat{f}(-\xi -k\theta) \, d \sigma_{r\eta}(\xi) \,dr \\
  \label{eq:resfin} &= (i \pi d + P)\left(  \frac{2}{|\eta|(1+r)} \int_{\Gamma_r(\eta)} e^{i  (-\xi-k\theta) \cdot x}  \, \widehat{f}(-\xi -k\theta) \, d \sigma_{r\eta}(\xi)\right).
\end{align} 

We recall that by \eqref{eq:Qjraw}, we have
\begin{equation} \label{eq:Qjraw2} 
 \widehat{Q_{j}(q)}(-2k\theta) =\int_{\RR^n} e^{ik \theta\cdot y} (qR_k)^{j-1}(q(\cdot)e^{ik\theta \cdot (\cdot)} )(y) \,dy.
 \end{equation}
Let $m \in \NN$, we define
\begin{equation} \label{eq:fm}
f_m(x) := R_k((q R_k)^{m-1}(q(\cdot) e^{ik\theta \cdot (\cdot)})) (x).
\end{equation}
We claim that
\begin{align}
  \nonumber f_m(x) &=\left( \prod_{i=1}^{m} (i\pi d_i + P_i) \right)   \left( \prod_{i=1}^{m}  \frac{2}{(1+r_i)} \right)  \frac{1}{|\eta|^{m}}\times \dots  \\
 \label{eq:fmex}  &\int_{\Gamma_{r_m}(\eta)}  \dots \int_{\Gamma_{r_1}(\eta)}e^{i (-\xi_m -k\theta) \cdot x} \,\widehat{q}(\eta -\xi_1) \left(\prod_{i=1}^{m-1} \widehat{q}(\xi_i - \xi_{i+1}) \right)   d\sigma_{\mb r}(\xi_1,\dots,\xi_{m}). 
 \end{align}
We prove the claim by induction. The case $m=1$ follows directly from \eqref{eq:resfin} using  that $\widehat{qe^{ik \theta \cdot (\cdot)}}(\xi)=
 \widehat{q}(\xi-k \theta)$ and that $\eta=-2k\theta$.
 
We are going to prove (\ref{eq:fmex}) for $m+1$  assuming that it is true for $m$. On the one hand, by (\ref{eq:fm}) and (\ref{eq:resfin}) we have 
\begin{align}\label{eq:parte}
\nonumber f_{m+1}(x) &= R_k(q f_{m} ) (x)=\left(i \pi d_{m+1}+P_{m+1}    \right)\dots \\
\hspace{3mm} &\left( \frac{2}{(1+r_{m+1})|\eta|}\int_{\Gamma_{r_{m+1}}(\eta)}e^{i(-\xi_{m+1}-k \theta)\cdot x}(\widehat{qf_m})(-\xi_{m+1}-k\theta) \, d\sigma_{r_{m+1}\eta}(\xi_{m+1})      \right).
\end{align}
On the other hand, by (\ref{eq:fmex}), changing the order of integration we have 
\begin{align}
\nonumber (\widehat{q f_m})(\zeta) &= \int_{\mathbb{R}^n}q(y)f_m(y)e^{-i\zeta\cdot y} \,dy  \\
\nonumber &=\left(  \prod_{i=1}^m (i \pi d_i+P_i  )  \right)  \left(\prod_{i=1}^m \frac{2}{1+r_i}    \right)\frac{1}{|\eta|^m}  \int_{\Gamma_{r_m}(\eta)}  \dots \int_{\Gamma_{r_1}(\eta)} \,\widehat{q}(\eta -\xi_1) \times \dots \\ 
&\label{eq:alberto_2} \hspace{30mm}  \left(\prod_{i=1}^{m-1} \widehat{q}(\xi_i - \xi_{i+1}) \right)\widehat{q}(k \theta+\zeta +\xi_m)   \, d\sigma_{\mb r}(\xi_1,\dots,\xi_{m})   .
\end{align}
Thus, putting  $\zeta =-\xi_{m+1}-k\theta$  in the previous equality and using  (\ref{eq:parte}),  we get
\begin{align*}
   f_{m+1}(x) &=\left( \prod_{i=1}^{m+1} (i\pi d_i + P_i) \right)   \left( \prod_{i=1}^{m+1}  \frac{2}{(1+r_i)} \right)  \frac{1}{|\eta|^{m+1}}\times \dots \\
    &\int_{\Gamma_{r_{m+1}}(\eta)}  \dots \int_{\Gamma_{r_1}(\eta)}e^{i (-\xi_{m+1} -k\theta) \cdot x} \,\widehat{q}(\eta -\xi_1) \left(\prod_{i=1}^{m} \widehat{q}(\xi_i - \xi_{i+1}) \right)   d\sigma_{\mb r}(\xi_1,\dots,\xi_{m+1}), 
 \end{align*}
which proves the claim.
 
By \eqref{eq:Qjraw2}  we have  that $\widehat{Q_j(q)}(-2k \theta)=\widehat{qf_{j-1}}(-k\theta)$, and
hence, in order to obtain (\ref{eq:Qj}), is enough to put  $\zeta=- k \theta $ in \eqref{eq:alberto_2},
\begin{align*}
\widehat{Q_j(q)}(-2k \theta) &=\left( \prod_{i=1}^{j-1}(i\pi d_i + P_i) \right)   \left( \prod_{i=1}^{j-1}  \frac{2}{(1+r_i)} \right)  \frac{1}{|\eta|^{j-1}} \times \dots \\
 &\int_{\Gamma_{r_{j-1}}(\eta)}  \dots \int_{\Gamma_{r_1}(\eta)}\widehat{q}(\eta -\xi_1) \left(\prod_{i=1}^{j-2}\widehat{q}(\xi_i - \xi_{i+1})\right) \widehat{q}( \xi_{j-1})  \, d\sigma_{\mb r}(\xi_1,\dots,\xi_{j-1}).
 \end{align*}
\end{proof}
 We  now introduce now the $\widetilde S_{j, \mb r}$ spherical  operators,
$$\widetilde S_{j,\mb r}(q)(\eta) := \chi(\eta) S_{j,\mb r}(q)(\eta),$$
as we did for $j=2$. The following proposition generalizes the results  of Lemma \ref{lemma:SrQ2} and Proposition \ref{prop:derSr} for  $j\ge 2$. Its proof will be given later on.
\begin{proposition} \label{prop:Srmulti}
Let $q \in \mathcal S(\RR^n)$, $n\ge 2$, $j\ge 3$ and $0<\delta<1$. Consider  all the multi-indices $\mb a = (a_1,\dots,a_{j-1})$ with $a_i$, $1\le i \le j-1$, either $0$ or $1$. Then the estimate
 \begin{equation} \label{eq:Srmultiest}
 \norm{ \partial_{\mb r}^{\mb a} \widetilde{S}_{j,\mb r} (q)}_{L^2_{\alpha-|\mb a|}} \le C \left( \prod_{i=1}^{j-1} \frac{1}{(1+r_i)^{\gamma}} \right) \norm{q}_{W^{\beta,2}_4}^j,
 \end{equation}
holds for  $\beta\ge0$, a certain  $\gamma >0$ (possibly dependent on $\beta$), and some constant $C =C(n,j,\alpha,\beta)$, if the following conditions also hold
\begin{equation} \label{eq:forgotten}
 r_i \in (1-\delta,1+\delta) \text{ if } a_i = 1, \text{ and } r_i \in (0,\infty) \text{ if } a_i=0.
\end{equation}
\begin{equation} \label{eq:rangeSphej}
  \begin{cases}
 \alpha \le  \beta + (j-1)(\beta - (n-3)/2), \hspace{4 mm} if \hspace{4mm} (n-3)/2 <\beta < (n-1)/2,\\
 \alpha < \beta + (j-1), \hspace{32mm} if \hspace{4mm} (n-1)/2 \le \beta<\infty .
         \end{cases}
\end{equation} 
\end{proposition}

With this proposition we  can prove finally Theorem \ref{teo.Qj}, with the help of the following density argument.
\begin{lemma} \label{lemma:density}
 Assume that the operator $\widetilde{Q}_j$ satisfies an a priori estimate
\begin{equation} \label{eq:estref}
\norm{\widetilde{Q}_j(q)}_{W^{\alpha,2}} \le C \norm{q}_{W^{\beta,p}_\delta}^{j},
\end{equation}
 for every $q\in C^\infty_c(\RR^n)$. Then there is a unique continuous extension  $\widetilde{Q}_j : W^{\beta,p}_\delta(\RR^n) \longrightarrow W^{\alpha,2}(\RR^n)$ of the operator, and estimate $(\ref{eq:estref})$ holds also for $q\in  W^{\beta,p}_\delta(\RR^n)$.
\end{lemma}
This lemma is just a trick to extend estimates for $\widetilde{Q}_j(q)$ without having to give an estimate for the  multilinear operator $Q_j(f_1,\dots,f_j)$ (this operator is defined by putting $f_i$ instead of $q$ in (\ref{eq:Qjraw}) following the order of appearance of each $q$ in the formula). It is a direct consequence of the more general statement given in {Lemma \ref{lemma:densityMain}} below.  
The key idea in the proof is to symmetrize $Q_{j}(f_{1},f_{2},\dots ,f_{j})$ and use a polarization identity for multilinear operators. (The advantage of having $f_{i} =q$ in most of the
estimates in this work is a question of  notational simplicity,
but it is not an essential restriction in any of them.)

\begin{proof}[Proof of  {Theorem \ref{teo.Qj}}]
We begin with the case $j=2$. By  {Proposition \ref{prop:derSr}} and  {Lemma \ref{lemma:SrQ2}} for each $q\in \mathcal S(\mathbb{R}^n)$ we can apply  {Proposition \ref{prop:genPV}} with $F_r =\widetilde S_r(q)$, $p=2$, $\tau = \alpha-1$ and $M=C \lVert q\rVert_{W_2^{\beta,2}}^2$. Therefore by   {(\ref{eq:Q2tilde})} this yields the estimate
\[  \lVert \widehat{\widetilde Q_2(q)}\rVert_{L^2_{\alpha'}} \le C \lVert  q \rVert_{W_2^{\beta,2}}^2,\]
  for $\alpha'<\alpha$  and $\alpha$ in the range   {(\ref{eq:range:S2})}.   Then by Plancherel theorem we get the desired estimate for $\widetilde Q_2(q)$ in the Sobolev norm, and by   {Lemma \ref{lemma:density}} we can extend by
density these estimates for $q\in W_{2}^{\beta ,2}(\mathbb{R}^{n})$.
This is enough to prove estimate  {(\ref{eq:estimateQ2})}.

Now, let's study the case $j\ge 3$. Consider $f\in \mathcal{S}$. We
introduce the following operators,
\begin{align}
\label{eq:T1}
T_{j,1}(r_{1},\dots ,r_{j-1})(f) :
&= \widetilde{S}_{j,\mathbf{r}}(f),
\\
\label{eq:Tkind}
T_{j,k}(r_{k}, \dots ,r_{j-1})(f) : &= (i\pi d_{k-1}
+P_{k-1})T_{j,k-1}(r_{k-1}, \dots ,r_{j-1})(f)
\\
\nonumber
&= \prod _{i=1}^{k-1} (i\pi d_{i} +P_{i})
\widetilde{S}_{j,\mathbf{r}}(f),
\end{align}
for $2 \le k \le j-1$,  and 
\begin{align} 
\nonumber T_{j,j}(f) := (i\pi d_{j-1} +P_{j-1})T_{j,j-1}(r_{j-1})(f)  &= \prod_{i=1}^{j-1} (i\pi d_i +P_i) \widetilde{S}_{j,\mathbf{r}}(f) \\
\label{eq:Tj} &= \widehat{\widetilde{Q}_j(f)}.
\end{align} 
$T_{j,k}(r_{k}, \dots ,r_{j-1})(f)(x)$ is a well
defined function, smooth in the variables $r_{k},\dots ,r_{j-1}$ and
$x$ (see  {Proposition \ref{prop:smooth}} in the appendix  for more
details). As we are going to see, the proof can be reduced to proving
the following claim.

\medskip
\textbf{Claim.} \emph{Let $1\le k\le j$, and let $\mathbf{a} = (a
_{1},\dots ,a_{j-1})$ with $a_{i} =0$ if $1\le i \le k-1$, and
$a_{i} =0,1$ if $k \le i \le j-1$. Then the estimate
\begin{equation}
\label{eq:Tkest}
\lVert \partial _{\mathbf{r}}^{\mathbf{a}} T_{j,k}(r_{k},\dots ,r_{j-1})(f)
\rVert _{L^{2}_{\alpha '}} \le c_{k} \lVert f\rVert ^{j}_{W^{\beta ,2}
_{4}},
\end{equation}
holds for $\alpha ' < (\alpha -|\mathbf{a}|)$ if conditions
\textup{ {(\ref{eq:forgotten})}} and \textup{ {(\ref{eq:rangeSphej})}} are
satisfied, with a constant $c_k$ given by
\begin{equation}
\label{eq:ck}
c_{k} = C C_{2}^{k-1}\prod _{i=k}^{j-1} \frac{1}{(1+r_{i})^{\gamma }},
\end{equation}
where $C_{2}$ is the constant introduced  in
 {Proposition \ref{prop:genPV}}.}
\medskip

By  {(\ref{eq:Tj})},  we have that for $q\in \mathcal S(\mathbb{R}^n)$, estimate  {(\ref{eq:Tkest})}  with $k=j$, $\mathbf{a}= 0$, and $f=q$ gives
\begin{equation*}
 \lVert \widehat{\widetilde Q_j(q)} \rVert_{L^2_{\alpha'}} \le C \lVert q \rVert_{W_2^{\beta,2}}^j,
\end{equation*} 
for every $\alpha'<\alpha$ and $\alpha$ in the range (\ref{eq:rangeSphej}).
Then, using Plancherel theorem and  {Lemma \ref{lemma:density}} to extend the resulting estimate for all $q \in W^{\beta,2}_4(\mathbb{R}^n)$, yields estimate (\ref{eq:estimateQj}). This is enough to conclude the proof of the theorem.

We now prove the claim by induction in $k$ (observe that $j$ is fixed in the claim). By  {(\ref{eq:T1})}, the case $k=1$
of estimate  {(\ref{eq:Tkest})} is equivalent to  {Proposition \ref{prop:Srmulti}}. To prove that  {(\ref{eq:Tkest})} holds true for
$2\le k \le j$, in each induction step we are going to use  {Proposition \ref{prop:genPV}} and  {(\ref{eq:Tkind})}. 

Let's assume that the claim holds for a certain $k$, $1\le k <j-1$, then
we are going to prove it for $k+1$. Let $\mathbf{a}' = (a'_{1},\dots
,a'_{j-1})$ with $a'_{i} =0$ if $1\le i \le k$, and $a'_{i} =0,1$ if
$k+1\le i \le j-1$. We are going to apply  {Proposition \ref{prop:genPV}} with
\begin{equation*}
F_{r_{k}}(x) := \partial _{\mathbf{r}}^{\mathbf{a}'} T_{j,k}(r_{k},
\dots ,r_{j-1})(f)(x).
\end{equation*}
By the induction hypothesis  {(\ref{eq:Tkest})} with $\mathbf{a} =
\mathbf{a}'$, and  {(\ref{eq:ck})} we have
\begin{equation}
\label{eq:Tk3ant}
\lVert F_{r_{k}}\rVert _{L^{2}_{\alpha '}} \le \frac{c_{k+1}}{(1+r
_{k})^{\gamma }} C_{2}^{-1} \lVert f\rVert ^{j}_{W^{\beta ,2}_{4}},
\end{equation}
for $\alpha ' < (\alpha -|\mathbf{a}'|)$ and $r_{k} \in (0,\infty )$.
Moreover, taking now $\mathbf{a}$ with $a_{i}=a'_{i}$ for $i \neq k$,
and $a_{k}=1$, we also get from  {(\ref{eq:Tkest})} the estimate
\begin{equation}
\label{eq:Tk3}
\lVert \partial _{r_{k}} F_{r_{k}} \rVert _{L^{2}_{\alpha '-1}}
\le \frac{c_{k+1}}{(1+r_{k})^{\gamma } } C_{2}^{-1} \lVert f\rVert
^{j}_{W^{\beta ,2}_{4}},
\end{equation}
with $\alpha ' < (\alpha -|\mathbf{a}'|-1)$ and $r_{k} \in (1-\delta
,1+\delta )$. Then, for each $f\in \mathcal S(\mathbb{R}^n)$ we can apply    {Proposition \ref{prop:genPV}}  since condition (\ref{eq:pv:Fr1}) is given by (\ref{eq:Tk3}) and   (\ref{eq:pv:Fr2}) by  (\ref{eq:Tk3ant}) with $M = c_{k+1}C_2^{-1} \lVert f \rVert^{j}_{W^{\beta,2}_4}$.
Therefore, for $\alpha ' < (\alpha -|\mathbf{a}'|)$, we obtain that
\begin{align*}
&\lVert \partial _{\mathbf{r}}^{\mathbf{a}'} T_{j,k +1}(r_{k},\dots ,r
_{j-1})(f)\rVert _{L^{2}_{\alpha '}} =
\\
&\lVert (i\pi d_{k} + P_{k})\partial _{\mathbf{r}}^{\mathbf{a}'} T
_{k}(r_{k},\dots ,r_{j-1})(f)\rVert _{L^{2}_{\alpha '}} = \lVert (i
\pi d_{k} + P_{k}) F_{r_{k}} \rVert _{L^{2}_{\alpha '}} \le c_{k}
\lVert f\rVert ^{j}_{W^{\beta ,2}_{4}},
\end{align*}
where the first equality is true by  {Proposition \ref{prop:smooth}} in the
appendix. This concludes the proof of the claim.
\end{proof}

We devote the remaining part of this section to prove Proposition  \ref{prop:Srmulti}. We define the operator
\begin{align*}
& K_{j,\mb r}(g_1,\dots, g_j)(\eta) =  K_{j,\mb r}(g_i)(\eta):=  \\
& \frac{1}{|\eta|^{j-1}} \int_{\Gamma_{\mb r}(\eta)}   |g_1(\eta -\xi_1)| \left(\prod_{i=1}^{j-2} |g_{i+1}(\xi_i - \xi_{i+1})| \right) |g(\xi_{j-1})| \,\, d\sigma_{\mb r}(\xi_1,\dots,\xi_{j-1}),
\end{align*} 
and $\widetilde K_{j,\mb r}(g_1,\dots, g_j)(\eta) := \chi(\eta)K_{j,\mb r}(g_1,\dots, g_j)(\eta)$ . Hence we have that
\begin{equation} \label{eq:KandS}
\left|\widetilde{S}_{j,\mb r}(q) (\eta)\right| \le  \left(\prod_{i=1}^{j-1} \frac{2}{1+r_i}\right) \widetilde K_{j,\mb r}(\widehat{q},\dots,\widehat{q})(\eta) .
\end{equation}
The main tool to prove Proposition \ref{prop:Srmulti} is the following Lemma.
\begin{lemma} \label{lemma:Kj}
Let $n\ge 2$ and $j\ge 3$, and consider  $f_l \in W_2^{\beta,2}(\RR^n)$, $1 \le l \le j $ with $\beta \ge 0$. Then the estimate
\begin{equation} \label{eq:Kjmain}
  \norm{\widetilde K_{j,\mb r}(\widehat{f_1},\dots,\widehat{f_j})}_{L^2_\alpha} \le C\left( \prod_{i=1}^{j-1} (1+r_i)^{\lambda} \right) \prod_{l=1}^{j} \norm{f_l}_{W_2^{\beta,2}},
\end{equation}
holds  when $\alpha$ is in the range given in $(\ref{eq:rangeSphej})$ for some real number $0<\lambda<1$. 
\end{lemma}
\begin{proof} 
 Since 
 $$\eta = (\eta-\xi_1) + \sum_{i=1}^{j-2}(\xi_i-\xi_{i+1}) + \xi_{j-1},$$ 
 if we fix $c<1/j$, one of the  conditions $|\eta-\xi_1|>c|\eta|$, $|\xi_i-\xi_{i+1}|>c|\eta|$ for some $1\le i\le j-2$, or $|\xi_{j-1}|>c|\eta|$  must  hold.  Hence, the sets
\begin{align*}
A^1_\textbf{r}(\eta) &= \left\{      (\xi_1, \dots , \xi_{j-1}) \in \Gamma_\textbf{r}(\eta): \; |\eta -\xi_1 |>c |\eta| \right\}, \\
A^i_\textbf{r}(\eta) &= \left\{ (\xi_1,  \dots  , \xi_{j-1}) \in \Gamma_\textbf{r}(\eta): \; |\xi_{i-1} -\xi_i |>c |\eta| \right\}, \;\; i=2,   \dots , j-1, \\
A^j_\textbf{r}(\eta) &=\left\{      (\xi_1,  \dots  , \xi_{j-1}) \in \Gamma_\textbf{r}(\eta): \; |\xi_{j-1} |>c |\eta| \right\},
\end{align*}
satisfy $\Gamma_\textbf{r}(\eta)= \cup_{k=1}^jA^k_\textbf{r}(\eta)$. As a consequence 
\begin{equation} \label{eq:allcases}
\|\widetilde{K}_{j,\textbf{r}}(\widehat{f_1}, \dots , \widehat{f_j})\|_{L^2_\alpha} \leq \sum_{k=1}^j \|\widetilde{K}^k_{j,\textbf{r}}(\widehat{f_1}, \dots , \widehat{f_j})\|_{L^2_\alpha} ,
\end{equation}
 where $\widetilde{K}^k_{j,\textbf{r}}$ is defined as $\widetilde{K}_{j,\textbf{r}}$ but integrating  over $A^k_\textbf{r}(\eta)$ instead of $\Gamma_{\bf r}(\eta)$.

We now fix a parameter $0<\lambda<(n-1)/2$, such that
\begin{equation} \label{eq:Kjbeta}
\beta := \alpha -(j-1)(1-\lambda).
\end{equation}  
In the region where $\chi(\eta)$ does not vanish, $|\eta| \sim \jp{\eta}$, and hence
\begin{align} 
  \nonumber &\|\widetilde{K}^k_{j,\textbf{r}}(\widehat{f_1}, \dots , \widehat{f_j})\|_{L^2_\alpha}^2  \le \\
   \label{eq:KjB} &\int_{\RR^n} \jp{\eta}^{2\beta}|\eta|^{-2(j-1)\lambda}
 \left( \int_{A^k_\textbf{r}(\eta)}   |\widehat{f_1}(\eta -\xi_1)| \left(\prod_{i=1}^{j-2}| \widehat{f}_{i+1}(\xi_i - \xi_{i+1})|  \right) |\widehat{f_j}(\xi_{j-1})| \, d\sigma_{\mb r} \right)^2 d\eta,
\end{align}
where  $d\sigma_{\mb r} = d\sigma_{\mb r}(\xi_1,\dots,\xi_{j-1})$.
The analysis is exactly the same for each $\widetilde{K}^k_{j,\textbf{r}}$ $k=1,\cdot \cdot \cdot, j$, so we only show one explicitly, for example the case  $k=j$.

 If $\beta \ge 0$ we can use that  ${\jp{\eta}^\beta \le C \jp{\xi_{j-1}}^\beta}$  in  $A^j_{\mb r}(\eta)$. Hence multiplying and dividing by $|\eta-\xi_1|^{n-1-2\gamma}  \prod_{i=1}^{j-2} |\xi_{i}-\xi_{i+1}|^{n-1-2\gamma}$ and applying Cauchy-Schwarz inequality, we get the following point-wise estimate for the integrand of (\ref{eq:KjB}),
\begin{align}
\nonumber \jp{\eta}^{2\beta} &|\eta|^{-2(j-1)\lambda} \left( \int_{A^j_{\mb r}(\eta)}   |\widehat{f_1}(\eta -\xi_1)| \left(\prod_{i=1}^{j-2}| \widehat{f}_{i+1}(\xi_i - \xi_{i+1})|   \right) |\widehat{f_j}(\xi_{j-1})| \, d\sigma_{\mb r} \right)^2 \\
 \nonumber \le |\eta|^{-2(j-1)\lambda} &\int_{\Gamma_{\mb r}(\eta)} |\widehat{f_1}(\eta -\xi_1)|^2|\eta-\xi_1|^{n-1-2\lambda} \left(\prod_{i=1}^{j-2}| \widehat{f}_{i+1}(\xi_i - \xi_{i+1})|^2|\xi_i - \xi_{i+1}|^{n-1-2\lambda} \right) \\
     \label{eq:Kj1}  \dots &\times |\widehat{f_j}(\xi_{j-1})|^2 \jp{\xi_{j-1}}^{2\beta} \, d\sigma_{\mb r}  \int_{\Gamma_{\mb r}(\eta)} \frac{1 }{|\eta-\xi_1|^{n-1-2\gamma}}  \prod_{i=1}^{j-2} \frac{1}{|\xi_{i}-\xi_{i+1}|^{n-1-2\gamma}}   \, d\sigma_{\mb r} .
\end{align}  
 Now, by  Lemma \ref{lemma.integrals}  we have that
\begin{equation} \label{eq:iterInteg}
|\eta|^{-2(j-1)\lambda} \int_{\Gamma_{\mb r}(\eta)} \frac{ 1 }{|\eta-\xi_1|^{n-1-2\gamma}} \prod_{i=1}^{j-2} \frac{1}{|\xi_{i}-\xi_{i+1}|^{n-1-2\gamma}}  \, d\sigma_{\mb r} \le C \prod_{i=1}^{j-1} r_i^{2\lambda},
\end{equation}
 where $C$ is some constant independent of $\eta$ (to see this, always compute first the integral  in the variable $\xi_{i}$ that only appears in one factor, in this case $\xi_{j-1}$). Hence, using this in (\ref{eq:Kj1}) and integrating in the $\eta$ variable we get the estimate 
 \begin{equation*}
\|\widetilde{K}^j_{j,\textbf{r}}(\widehat{f_1}, \dots , \widehat{f_j})\|_{L^2_\alpha}^2 \le C\prod_{i=1}^{j-1} r_i^{2\lambda}  \int_{\RR^n} \int_{\Gamma_{\mb r}(\eta)}  |F(\xi_1,\dots,\xi_{j-1},\eta)|^2 \, d\sigma_{\mb r} \,d\eta,
\end{equation*} 
where
\begin{align*} 
  F(\xi_1,\dots,&\xi_{j-1},\eta) := \widehat{f_1}(\eta -\xi_1) \jp{\eta-\xi_1}^{(n-1)/2-\lambda} \times \dots \\
 &\nonumber  \left(\prod_{i=1}^{j-2} \widehat{f}_{i+1}(\xi_i - \xi_{i+1})\jp{\xi_i - \xi_{i+1}}^{(n-1)/2-\lambda} \right) \widehat{f_j}(\xi_{j-1}) \jp{\xi_{j-1}}^{\beta} .
 \end{align*}
Therefore,  as in Lemma \ref{lemma.K.Q2}, we apply  the trace theorem to each one of the integrals over he spheres $\Gamma_{r_i}(\eta)$, to obtain
\begin{align} 
\nonumber \|\widetilde{K}^j_{j,\textbf{r}}&(\widehat{f_1}, \dots , \widehat{f_j})\|_{L^2_\alpha}^2  \le C \left(\prod_{i=1}^{j-1} r_i^{2\lambda} \right) \times \dots \\
 \label{eq:Kj3}  &\hspace{-4mm}\sum_{0 \le |\alpha_1|,\dots, |\alpha_{j-1}| \le 1} \int_{\RR^n} \dots \int_{\RR^n} | \partial_{\xi_1}^{\alpha_1} \dots \partial_{\xi_{j-1}}^{\alpha_{j-1}} F(\xi_1,\dots,\xi_{j-1},\eta)|^2 \,d \xi_1 \dots \,d\xi_{j-1}  \, d\eta,
\end{align}
 where the $\alpha_i$ are multi-indices related to the corresponding $\RR^n$ variables $\xi_i$, see Lemma \ref{lemma.trazas2} in the appendix for a more detailed formulation.  Now, using the Leibniz rule in (\ref{eq:Kj3}) we can  put the derivative operators on the functions ${\widehat{f_i}\jp{\cdot}^{a}}$. In the worst case we are going to get  terms  of the kind
  $$\partial_{\xi_i}^{\alpha_i} \partial_{\xi_{i+1}}^{\alpha_{i+1}} \left( \widehat{f}_{i+1}(\xi_{i} -\xi_{i+1})\jp{\xi_{i} -\xi_{i+1}}^{a} \right),$$
  with at most two derivative operators having $|\alpha_i| = |\alpha_{i+1}|=1$. Therefore, if we  integrate each  summand first in $\eta$ and then in $\xi_1, \xi_{2}, \dots ,\xi_{j-1}$, we obtain
\begin{equation} \label{eq:Kj4}
\|\widetilde{K}^j_{j,\textbf{r}}(\widehat{f_1}, \dots , \widehat{f_j})\|_{L^2_\alpha}^2  \le C \left(\prod_{i=1}^{j-1} r_i^{2\lambda} \right) \norm{f_{j}}_{W_2^{\beta,2}}^2\prod_{l=1}^{j-1} \norm{f_l}_{W_2^{(n-1)/2 - \lambda ,2}}^2 .
\end{equation}
 Putting together (\ref{eq:Kj4}) and the analogous estimates coming from the analysis of the other cases in (\ref{eq:allcases}), we obtain
\begin{align*}
\norm{\widetilde{K}_{j,\mb r}(\widehat{f_i})}_{W_2^{\alpha,2}} \le C \left( \prod_{i=1}^{j-1} r_i^{2\lambda} \right) \sum_{i=1}^{j} \, \, \norm{f_i}_{W_2^{\beta,2}} \prod_{\substack{1\le l \le j \\ l \neq i}} \norm{f_l}_{W_2^{(n-1)/2-\gamma,2}} .
\end{align*}
We reason as in Lemma \ref{lemma.K.Q2}. First we impose the extra condition  $\lambda<1$.

As a consequence of Remark \ref{remark:Sob}, equation (\ref{eq:Kjmain}) follows directly in the range $\beta \ge (n-1)/2$. The restrictions on $\lambda$ together with (\ref{eq:Kjbeta}) give us the following restrictions on $\alpha$,
\begin{equation}\label{restriccion_2}
\begin{cases}
0 < \lambda < 1 \\ 0 < \lambda \le \frac{n-1}{2}
\end{cases}   \Longleftrightarrow \begin{cases}
\beta +(j-2)< \alpha < \beta +(j-1)  \\ \beta +(j-1) - \frac{n-1}{2} \le \alpha < \beta +(j-1).
\end{cases}                
\end{equation}
We discard the lower bounds for $\alpha$ as in Lemma \ref{lemma.K.Q2}.

Otherwise, if $\beta$ is in the range $\beta < (n-1)/2$, estimate (\ref{eq:Kjmain}) holds if we add the extra condition  $(n-1)/2-\lambda \le \beta$. Then, since $\lambda<1$, we must have $\beta>(n-3)/2$ as in Lemma \ref{lemma.K.Q2}. Also, (\ref{eq:Kjbeta}) together with $(n-1)/2-\lambda \le \beta$ imply that $\alpha \le \beta + (j-1)(\beta-(n-3)/2)$ which is a more restrictive condition than $\alpha<\beta+(j-1)$ since we have $\beta<(n-1)/2$. Hence, we have obtained the ranges of parameters given in the statement.
\end{proof}

\begin{proof}[Proof of Proposition $\ref{prop:Srmulti}$]
By (\ref{eq:KandS}), estimate (\ref{eq:Srmultiest}) follows directly if $\mb a = 0$. Therefore, we consider the case $\mb a \neq 0$.

Let $\mb r = (r_1,\dots,r_k)$, $1\le k < \infty$. We follow the same computations  to obtain  (\ref{eq:derivative2}) from   (\ref{eq:derivative1}). We have that for a general function $F(\xi_1,\dots,\xi_k,\eta)$, $C^1$ in the first $k$ variables, 
\begin{align} 
\nonumber \partial_{r_i} &\left( \frac{1}{1+r_i} \int_{\Gamma_{\mb r}(\eta)}  F(\xi_1,...,\xi_k,\eta) \,\, d\sigma_{\mb r}(\xi_1,\dots,\xi_k) \right)  \\
 \nonumber &  \hspace{5mm} =\frac{(n-2)r_{i}+(n-1)}{r_{i}(1+r_{i})^2} \int_{\Gamma_{\mb r}(\eta)}  F(\xi_1,...,\xi_k,\eta) \,\, d\sigma_{\mb r} \\
\label{eq:lastlemm} & \hspace{15mm} + |\eta|\frac{1}{1+r_{i}}\int_{\Gamma_{\mb r}(\eta)}  \theta_i \cdot \nabla_{\xi_i} F(\xi_1,...,\xi_k,\eta) \,\, d\sigma_{\mb r}, 
\end{align}
where $\theta_i = \frac{\xi_i -\eta/2}{|\xi_i -\eta/2|}$ is a unitary vector. Observe that the coefficients before the integrals are functions of $r_i$ which are bounded for $r_i \in (1-\delta,1+\delta)$ for any $0<\delta<1$ fixed. Hence if we take a derivative $\partial^{\mb a}_{\mb r}$ with $\mb a= (a_1,\dots,a_k)$ and  $a_i=0,1$,   we have
\begin{align}
 \nonumber &\left|\partial_{\mb r}^{\mb a} \left( \left( \prod_{i=1}^{j-1} \frac{1}{1+r_i} \right) \int_{\Gamma_{\mb r}(\eta)}  F(\xi_1,...,\xi_k,\eta) \, d\sigma_{\mb r} \right)\right|  \\
 \label{eq:d1} &\le C \left( \prod_{i=1}^{j-1} \frac{1}{1+r_i} \right)  \sum_{ \substack{0\le |\alpha_i| \le a_i, \\ 1 \le i \le k }} |\eta|^{|\alpha_1| + \dots + |\alpha_k| }\int_{\Gamma_{\mb r}(\eta)} \left| \partial_{\xi_1}^{\alpha_1} \dots \partial_{\xi_k}^{\alpha_k} F(\xi_1,...,\xi_k,\eta) \right| \, d\sigma_{\mb r},
 \end{align}
 where $\alpha_i$ are multi-indices associated to derivatives in $\RR^n$, and we have imposed $r_i \in (1-\delta , 1+\delta)$ if $a_i=1$ (to bound the coefficients dependent on $r_i$ as we did before). Notice that  $|\alpha_1| + \dots + |\alpha_k|$ can take all the integer values from $0$ to $|\mb a|$. We are interested in  computing $\partial_{\mb r}^{\mb a} S_{j,\mb r}$ so we put $k=j-1$ and
\begin{equation*}   
 F(\xi_1,...,\xi_{j-1},\eta) =\widehat{q}(\eta -\xi_1) \left(\prod_{i=1}^{j-2} \widehat{q}(\xi_i - \xi_{i+1}) \right) \widehat{q}(\xi_{j-1}). 
 \end{equation*}
In this case,  each potential is going to be derived at most twice, since in the worst case $\widehat{q}$ is valued on the difference of two variables, $\xi_i$ and $\xi_{i+1}$. Therefore  it suffices to give the following rough estimate,
\begin{equation*}
\int_{\Gamma_{\mb r}(\eta)} \left| \partial_{\xi_1}^{\alpha_1} \dots \partial_{\xi_{j-1}}^{\alpha_{j-1}} F(\xi_1,...,\xi_{j-1},\eta) \right| \, d\sigma_{\mb r} \le C \sum_{\substack{0 \le |\alpha_i'| \le 2 \\ 1\le i\le j-1}} K_{j,\mb r}(\partial^{\alpha_1'} \widehat{q},\partial^{\alpha'_2} \widehat{q}, \dots, \partial^{\alpha'_{j-1}} \widehat  q) (\eta),
\end{equation*}
 for some new multi-indices $\alpha_1',\dots ,\alpha_{j-1}'$. Hence, by (\ref{eq:Sj}), putting together  (\ref{eq:d1}) and the previous estimate, we   obtain
\begin{align*} 
  &\left| \partial_{\mb r}^{\mb a}{S}_{j,\mb r} (q)(\eta) \right| \\
   &\le C \left( \prod_{i=1}^{j-1} \frac{1}{1+r_i} \right) \frac{1}{|\eta|^{j-1}} \left( 1+|\eta|+\dots+ |\eta|^{|\mb a|} \right) \sum_{\substack{0 \le |\alpha_i| \le 2 \\ 1\le i\le j-1}} K_{j,\mb r}(\widehat{ qx^{\alpha_1}},\widehat{ qx^{\alpha_2}}, \dots, \widehat{ qx^{\alpha_{j-1}}})(\eta).
  \end{align*}
Then,  since $|\alpha_i|\le 2$, multiplying the previous inequality by $\chi(\eta)$, (\ref{eq:Srmultiest}) follows form Lemma \ref{lemma:Kj} using that $\norm{x^{\alpha_i} q}_{W_2^{\beta,2}} \le C  \norm{q}_{W_4^{\beta,2}}$, which can be obtained by the the same reasoning given after  (\ref{eq:pain}).
\end{proof}


\section{Implicit estimates for the \texorpdfstring{$Q_j$}{Qj} operator} \label{sec:implicitQj}

In this section we prove Proposition \ref{prop:convergence}.
 We follow the method developed for fixed angle scattering in \cite{R},  and for backscattering in \cite{RV} (case $\beta \le n/2$) and \cite{RRe} (extension to general $\max(0,m)<\beta<\infty$). It has been also adapted to the elasticity setting in \cite{BFPRM1} and \cite{BFPRM2}.  As mentioned in the introduction, we have improved the regularity gain given in \cite{RV} and \cite{RRe} and we have obtained directly the general case $m \le  \beta<\infty$ by using the cancellation given by the fractional Laplacian $(-\Delta)^{s}$. As in the mentioned works, we begin by giving some estimates of the resolvent of the Laplacian (see  \cite{R,ReT} or \cite{notasR}). We define the conjugate resolvent operator
\begin{equation} \label{eq.Res.conj}
R_\theta(q)(x) := e^{-ik\theta \cdot x}R_k \left( e^{ik\theta\cdot(\cdot)} q(\cdot)\right )(x).
\end{equation}
\begin{lemma} \label{lemma.Res}
Let $s\ge 0$ and let $r$ and $t$ be such that $0\le 1/t-1/2 \le 1/(n+1)$ and $0\le 1/2 -1/r \le 1/(n+1)$. There exist $\delta$, $\delta'>0$ and $C$ (independent of k) such that
$$ \norm{R_{\theta}(q)}_{W^{s,r}_{-\delta}} \le C k^{-1 + (1/t-1/r)(n-1)/2  }\norm{q}_{W^{s,t}_{\delta'}} .$$
\end{lemma}
We need also a theorem of Zolesio on the product of functions in the Sobolev spaces (a proof can be found in \cite{grin} and for the compactly supported case in \cite[pp. 182-183]{ReT})
\begin{lemma}[Zolesio] \label{lemma.zol}
Let $s_1,s_2,s \ge 0$, $s \le s_1$, $s \le s_2$, and let $r,t$ and $p$ be such that $t < \min(p,r)$ and 
$$s_1+s_2-s \ge n \left( \frac{1}{p} + \frac{1}{r} - \frac{1}{t} \right).$$
Then
$$ \norm{q f}_{W^{s,t}} \le C \norm{q}_{W^{s_1,p}} \norm{f}_{W^{s_2,r}} .$$
Moreover, if $q$ is compactly supported and $\delta,\delta'\in \RR$, then
\begin{equation} \label{eq:zol}
 \norm{q f}_{W^{s,t}_{-\delta}} \le C(\supp \, q,\delta,\delta') \norm{q}_{W^{s_1,p}} \norm{f}_{W^{s_2,r}_{\delta'}} .
 \end{equation}
\end{lemma}
\begin{proof}[Proof of Proposition $\ref{prop:convergence}$]
For brevity, we will omit the dependency of the constants on the dimension $n$. Without loss of generality assume $q \in C^\infty_c(B_\rho)$, where $B_\rho$ denotes the ball of radius $\rho$. In terms of $R_\theta$, defined in (\ref{eq.Res.conj}),  the expression of $Q_j$ given in (\ref{eq:Qjraw}) becomes 
$$\widehat{Q_j(q)}(\xi) = \int_{\RR^n} e^{i2k\theta \cdot y} (q R_{\theta})^{j-1}(q)(y) \,dy,$$
with $\xi=-2k\theta$. In spherical coordinates we can write 
\begin{equation*}
\norm{\widetilde Q_j(q)}^2_{W^{\alpha ,2}} \le  \int_{C_0}^{\infty}  k^{n-1+2\alpha}\int_{\SP^{n-1}}\left | \int_{\RR^n} e^{i2k\theta \cdot y} (q R_{\theta})^{j-1}(q)(y) \,dy \right|^2 \,d\sigma(\theta) \, dk.
\end{equation*}
Now,  if $f$ is a $C_c^\infty(\RR^n)$ function and $\beta \ge 0$,  the fractional Laplacian (see for example \cite[Section 3]{valdinoci}) can be defined by the identity
\begin{equation} \label{eq:fracfourier}
 \mathcal{F} \left((-\Delta)^{\beta/2} \,f\right) (\xi) := |\xi|^\beta \, \widehat{f}(\xi),
 \end{equation}
and we have that in the sense of distributions 
$$(-\Delta)^{\beta/2} e^{i2k\theta \cdot x} = (2k)^{\beta} e^{i2k\theta \cdot x},$$
see \cite[chapter 2]{silvestre} for a rigorous extension to distributions of the fractional Laplacian. Hence applying this to the previous inequality, since $(q R_{\theta})^{j-1}(q)\in C^{\infty}_c(\RR^n)$ we obtain 
\begin{align}
\nonumber \norm{\widetilde Q_j(q)}^2_{W^{\alpha,2}}& \ \\ 
\nonumber \le C(\beta) \int_{C_0}^{\infty}  &k^{n-1+2\alpha-2\beta} \int_{\SP^{n-1}} \left | \int_{\RR^n} (-\Delta)^{\beta/2} (e^{i2k\theta \cdot y} )(q R_{\theta})^{j-1}(q)(y) \,dy \right|^2 \,d\sigma(\theta) \, dk \\
\nonumber  = C(\beta) \int_{C_0}^{\infty}  &k^{n-1+2\alpha-2\beta} \int_{\SP^{n-1}} \left | \int_{\RR^n}  e^{i2k\theta \cdot y} (-\Delta)^{\beta/2} \left( (q R_{\theta})^{j-1}(q)\right) (y) \,dy \right|^2 \,d\sigma(\theta) \, dk \\
\label{eq.prpQj} &\le C(\beta) \int_{C_0}^{\infty}  k^{n-1+2\alpha-2\beta} \int_{\SP^{n-1}} \norm{(-\Delta)^{\beta/2} \left( (q R_{\theta})^{j-1}(q)\right) }_{L^1(\RR^n)}^2 \, d\sigma(\theta) \, dk.
\end{align}
Applying Lemma \ref{lemma.LapFracL1} in the Appendix we have
\begin{align*}
\norm{(-\Delta)^{\beta/2} \left( (q R_{\theta})^{j-1}(q)\right)  }_{L^1} &\le C(\beta)\norm{\jp{\cdot}^{\delta}q}_{W^{\beta,2}}\norm{\jp{\cdot}^{-\delta}  R_\theta ((q R_{\theta})^{j-2}(q))}_{W^{\beta,2}}\\
&\le C(\beta,\rho)\norm{q}_{W^{\beta,2}}\norm{  R_\theta ((q R_{\theta})^{j-2}(q))}_{W_{-\delta}^{\beta,2}}
\end{align*}
using that $q$ is compactly supported. Now, choose  $\delta$ in the previous equation as in Lemma \ref{lemma.Res}.
The idea to deal with the norm in the right hand side  is to iterate lemmas \ref{lemma.Res} and \ref{lemma.zol} following the diagram,

{\footnotesize{
\begin{equation*}
\minCDarrowwidth10mm
\begin{CD}
\mathrm{W}_{\delta'}^{\beta,t_{j-1}}
@>
\mathrm{R_\theta}
>>
\mathrm{W}_{-\delta}^{\beta,r_{j-1}}
@>
{\mathrm{\rm{q \cdot}}}
>>
\mathrm{W}_{\delta'}^{\beta,t_{j-2}}
\ldots 
@>
q\cdot
>>
W_{\delta'}^{\beta,t_{1}}  
@>
\mathrm{R_\theta}
>>
W_{-\delta}^{\beta,r_{1}}
\\
q
@>
{}
>>
R_\theta(q)
@>
{}
>>
qR_\theta(q)
\ldots
@>
{}
>>
(qR_\theta)^{j-2}(q)
@>
{}
>>
R_\theta(qR_\theta)^{j-2}(q)
\end{CD}
\end{equation*}
}}

where $r_1 =2$ and $t_{j-1} =2$ and $r_\ell$ and $t_\ell$, $\ell= 1, \dots, j-2$ have to satisfy  the conditions

\begin{align*}
&0\le \frac{1}{t_\ell} -\frac{1}{2} \le \frac{1}{n+1} \hspace{8mm}\text{and} \hspace{8mm} 0\le \frac{1}{2} -\frac{1}{r_{\ell+1}} \le \frac{1}{n+1},\\
&t_\ell < 2 \hspace{29.5mm} \text{and} \hspace{8mm} 0\le \frac{1}{2} +\frac{1}{r_{\ell+1}} - \frac{1}{t_\ell} \le  \frac{\beta}{n}.
\end{align*}
Hence we obtain
$$\norm{  R_\theta ((q R_{\theta})^{j-2}(q))}_{W_{-\delta}^{\beta,2}} \le C^j(\beta,\rho) k^{\gamma_j}\norm{q}^{j-1}_{W^{\beta,2}},$$
(in (\ref{eq:zol}) the constant depends on the support $B_\rho$ of $q$) where
\begin{align*}
\gamma_j &= -(j-1) + \frac{(n-1)}{2}  \sum_{\ell=1}^{j-1} \left(\frac{1}{t_\ell}- \frac{1}{r_\ell} \right) \\
&=  -(j-1) + \frac{(n-1)}{2} \sum_{\ell=1}^{j-2} \left(\frac{1}{t_\ell}- \frac{1}{r_{\ell+1}}\right).
\end{align*}
Now, for small $\varepsilon > 0$, when $\beta\ge m = (n-4)/2 + 2/(n+1)$ we can choose $r_\ell$ and $t_\ell$ satisfying all the previous conditions and 
$$1/t_\ell - 1/r_{\ell+1} = \max (1/2-\beta/n,\varepsilon) ,$$
for all $1\le \ell \le j-2$, and so we obtain
$$\gamma_j = -(j-1) + \frac{(n-1)}{2}(j-2) \max (1/2-\beta/n,\varepsilon). $$
 Putting all the previous estimates together in (\ref{eq.prpQj}) we obtain
\begin{multline} \label{eq:Qjimplicit}
 \norm{\widetilde Q_j(q)}^2_{W^{\alpha ,2}} \le C^{2j}(\beta,\rho)\, \norm{q}^{2j}_{W^{\beta ,2}}\int_{C_0}^\infty k^{n-1 + 2\alpha -2\beta + 2\gamma_j} \, dk \\
 = C^{2j}(\beta,\rho) \frac{C_0^{-2(\alpha_j-\alpha)}}{\alpha_j-\alpha} \norm{q}_{ W^{\beta,2}}^{2j},
 \end{multline}
with $\alpha<\alpha_j$ and $\alpha_j = \beta + (j-1) - \frac{n}{2} - \frac{(n-1)}{2} (j-2) \max{\left( 0,\frac{1}{2} - \frac{\beta}{n} \right )}$. By density, we can extend estimate  (\ref{eq:Qjimplicit})  for  $q\in W^{\beta,2}(\RR^n)$ compactly supported in $B_\rho$. This follows from Lemma \ref{lemma:density}, with minor changes to take into account the restriction in the support.

Hence we can consider now $q\in W^{\beta,2}(\RR^n)$ compactly supported in $B_\rho$. Choose some $\alpha>0$. Now, for  $\beta \ge m$, $\alpha_j$ grows linearly with $j$. Then, for any integer $l>0$ such that  $\alpha_l > \alpha$  we have by the previous estimates that
$$\left \| \sum_{j=l}^{\infty}{\widetilde Q_{j}(q)} \right\|_{W^{\alpha,2}} \le \sum_{j=l}^{\infty} \norm{\widetilde Q_j(q)}_{W^{\alpha,2}}\le  \sum_{j=l}^{\infty} C_0^{-(\alpha_j-\alpha)} C^j(\alpha,\beta,\rho) \norm{q}_{W^{\beta,2}}^{j}.$$
Using  the linear growth of $\alpha_j$ we can choose some $\varepsilon(\beta) =\varepsilon>0$ such that for every $j\ge l$, $(\alpha_j-\alpha) \ge j \varepsilon$. Therefore we obtain  that
$$ \left \| \sum_{j=l}^{\infty}{\widetilde Q_{j}(q)} \right\|_{W^{\alpha,2}} \le   \sum_{j=l}^{\infty} C_0^{-\varepsilon j } C^j(\alpha,\beta,\rho)\norm{q}_{W^{\beta,2}}^{j},$$
and the right hand side converges taking $C_0 > \left( C(\alpha,\beta,\rho)\norm{q}_{W^{\beta,2}} \right)^{1/\varepsilon}$. 
\end{proof}


\section{Some limitations on the regularity of the double dispersion operator} \label{sec:5}

In this section we use a certain family of compactly  supported radial and real  functions, to obtain the upper bounds to the maximum regularity of the $ Q_2$ operator given by Theorem \ref{teo:Q2count}. This family of functions was constructed  in \cite{fix} to illustrate an analogous  phenomenon in the fixed angle and full data scattering problems. 

See also \cite[pp. 20]{BM10} for an explicit radial counterexample for the the case $\beta =(1/2)^-$ and $n=3$.
\begin{lemma}[Proposition 5.3  of \cite{fix}] \label{lemma:gbeta}
For every $0<\beta<\infty$ there is a radial, real and compactly supported function $g_\beta$  such that $\widehat{g_\beta}$ is non negative, $\widehat{g_\beta}(0)>0$, and for some $c>0$ we have that
\begin{equation} \label{eq:gbetasymp}
\widehat{g_\beta}(\xi) \sim \, \jp{\xi}^{-n/2-\beta} \, \, \, \text{if} \, \, \, |\xi|>c.
\end{equation}
\end{lemma}
Notice that $g_\beta \in W^{\gamma,2}$ if and only if $\gamma<\beta$.
The construction of these functions is not difficult, the idea is to define 
\begin{equation} \label{eq:lasthope}
g_\beta(x) :=  (\phi *\phi)(x) G_\beta(x) ,
\end{equation} 
where $\phi$ is any real and radial $C^\infty_c(\RR^n)$ function and  the  $G_\beta$ functions  are, up to normalizing factors,  kernels of Bessel potential operators. Indeed, if we have that
$$
\widehat{G_\beta}(\xi) := {\jp{\xi}^{-n/2-\beta}},
$$
then $G_\beta(x)$ is a smooth and exponentially decaying function outside the origin (see, for example \cite[Chapter V]{stein}). Hence,  multiplying $G_\beta$ as in (\ref{eq:lasthope}) by a $C^\infty_c(\RR^n)$ cut off function, non vanishing at the origin, we get the desired asymptotic behaviour  of $\widehat{g_\beta}(\xi)$. The choice of the cut off $\phi *\phi$ guarantees the positivity of $\widehat{g_\beta}(\xi)$ (see \cite{fix} for more details).

 The key idea behind the proof of Theorem \ref{teo:Q2count} is to  study the asymptotic behavior of $|\widehat{ {Q}_2(g_\beta)}(\eta)|$ when $|\eta| \to \infty $.  This is greatly simplified by the fact we have the explicit formula (\ref{eq:Q2}). Now,  $g_\beta$ has a real Fourier transform $\widehat{g_\beta}(\xi)$ by construction,  so $\widehat{Q_2(g_\beta)}$ has a real part given by the principal value term in (\ref{eq:Q2}) and an imaginary part given by $\pi S_{r=1}(g_\beta)$. As there is no possible cancellation between the real an imaginary parts, we are going to study only the asymptotic behavior of the spherical integral, which has the advantage of having a positive integrand.
To simplify notation we  put $S(q) := S_{r=1}(q)$ and $ \Gamma(\eta) := \Gamma_{r=1}(\eta)$. The main estimate is the following one.
\begin{lemma} \label{lemma.example.backs}
Let $\beta > -n/2$ and assume that $q_\beta \in \mathcal S'(\RR^n)$ satisfies the following conditions,
\begin{enumerate}[i)]
\item Its Fourier transform $\widehat{q_\beta}(\xi)$ is real and non negative function in all $ \RR^n$.

\item There is a constant $c>0$ such that if $|\xi|>c$, then $\widehat{q_\beta}(\xi)\ge C\jp{\xi}^{-n/2-\beta}$.

\item  $\widehat{q_\beta}(\xi)$ is  continuous and satisfies $\widehat{q_\beta}(0) >0$.
\end{enumerate}
Then we  have that if $|\eta|>4c$, there is a constant $C$ independent of $\eta$ such that
\begin{equation*}
S(q_\beta)(\eta) \ge C  \max\left( \jp{\eta}^{-\beta-n/2 -1},\jp{\eta}^{-2\beta-2}\right )  .
\end{equation*}
\end{lemma}

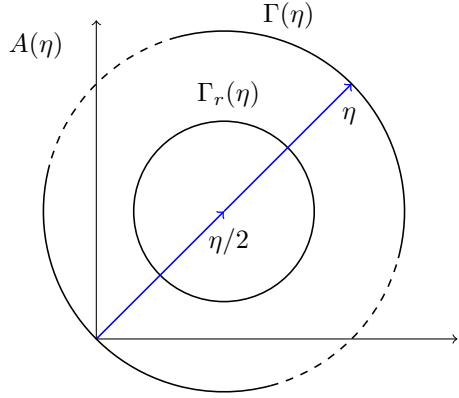
\begin{figure} 
\centering
\begin{tikzpicture}[scale=0.8]
 \draw[<->] (6,0) -- (0,0) --
(0,5.3);
\draw [  semithick] (2.12132,2.12132) circle [radius=1.5];
  \draw [->][semithick,blue] (0,0) -- (2.12132,2.12132);
    \draw [->][semithick,blue] (2.12132,2.12132) -- (4.24264,4.24264);
    \draw [dashed,semithick] ([shift=(105:3)]2.12132,2.12132) arc  (105:165:3);
        \draw [semithick] ([shift=(165:3)]2.12132,2.12132) arc  (165:285:3);   
         \draw [dashed,semithick] ([shift=(285:3)]2.12132,2.12132) arc  (285:345:3);
             \draw [semithick] ([shift=(-15:3)]2.12132,2.12132) arc  (-15:105:3);
 \node [above] at (2.2,3.7) {  $\Gamma_r(\eta)$ };
 \node [above] at (3.2,5) {  $\Gamma(\eta)$ };
 \node [below] at (2.2,2) {  $\eta/2$ };
 \node [below] at (4.2,4) {  $\eta$ };
  \node [above] at (-1,4.5) {  $A(\eta)$ };
\end{tikzpicture}
   \caption{The largest sphere is the Ewald sphere $\Gamma(\eta):= \Gamma_1(\eta)$, and the small one represents the Ewald sphere  $\Gamma_r(\eta)$ for some $r<1$.  The dashed region is the set $A(\eta)\subset \Gamma(\eta)$ .}
   
   \label{fig.ewald1}
\end{figure}

\begin{proof} 
Since $\widehat{q_\beta}$ is non negative, we have that
\begin{equation} \label{eq.example.3}
 S(q_\beta)(\eta) \ge   \frac{1}{|\eta|}\int_{A(\eta)}\widehat{q_\beta}(\xi)\widehat{q_\beta}(\eta-\xi) \,d\sigma_\eta(\xi),
\end{equation}
where, if we write $\eta = |\eta|\theta$ with $\theta$ a unitary vector, $A(\eta)\subset \Gamma(\eta)$ is defined  as follows
$$A(\eta) := \{ \xi \in \Gamma(\eta): |(\xi-\eta/2)\cdot \theta|\le |\eta|/4 \}.$$
That is, $A(\eta)$ is a band around the equator orthogonal to $\eta$ of width proportional  to $|\eta|$ (see figure \ref{fig.ewald1}). Observe that we have that $\xi\in A(\eta)$  if and only if $\eta-\xi\in A(\eta)$, and that in this region  $|\xi| \ge |\eta|/4$. Hence, if we consider $|\eta|>4c$ (where $c$ is given in the statement) and $\xi\in A(\eta)$, we have that $|\xi|>c$ and $|\eta-\xi|>c$, so from (\ref{eq.example.3}) we get
\begin{align}
\nonumber S(q_\beta)(\eta) &\ge C\frac{1}{|\eta|}\int_{A(\eta)} \jp{\eta-\xi}^{-\beta-n/2}\jp{\xi}^{-\beta-n/2} \,d\sigma_\eta(\xi)\\
  \label{anadida} &\ge  C \jp{\eta}^{-2\beta-n } |\eta|^{n-2} > C \jp{\eta}^{-2\beta -2},
\end{align}
where to get the last line we have used that the measure of $A(\eta)$ is proportional to  $|\eta|^{n-1}$, and that $|\xi| \le |\eta|$ and $|\eta-\xi| \le |\eta|$ always hold in $\Gamma(\eta)$.

 Now, if $\widehat{q_\beta}$ is continuous  and $\widehat{q_\beta}(0)>0$, we can take a ball $B_\varepsilon$ around the origin of radius $0 < \varepsilon< c$  such that $\widehat{q_\beta}{(\xi)}$ is positive in its closure. Then if $|\eta|> 2c$, $\xi \in B_\varepsilon\cap \Gamma(\eta)$ implies $|\eta-\xi| > c$, so
\begin{align}
\nonumber S(q_\beta)(\eta) &\ge  \frac{1}{|\eta|}\int_{B_\varepsilon\cap \Gamma(\eta)}\widehat{q_\beta}(\xi)\widehat{q_\beta}(\eta-\xi) \,d\sigma_\eta(\xi)\\
\label{eq.example.4} &\ge C \frac{1}{|\eta|}\int_{B_\varepsilon\cap\Gamma(\eta)} \jp{\eta-\xi}^{-\beta-n/2} \,d\sigma_\eta(\xi) \ge  C \jp{\eta}^{-\beta-n/2 -1} ,
\end{align}
using that  $|\eta-\xi| \le |\eta|$ always, and that the measure $|B_\varepsilon\cap\Gamma(\eta)|$ is bounded below by a positive constant independent of $\eta$ (this is because the region $B_\varepsilon\cap\Gamma(\eta)$ approaches a flat disc of radius $\varepsilon$ for $\eta$ large). To finish we have just to put together (\ref{anadida}) and (\ref{eq.example.4}).
\end{proof}
The proof of Theorem \ref{teo:Q2count} follows from the previous lemma and  the following simple result.
\begin{lemma} \label{lemma.Wloc}
Let $f \in \mathcal S'(\RR^n)$ be such that $\widehat f$ is a non negative measurable function. Assume also that for some $c>0$, $\gamma \in \RR $ and $|\eta|>c$ we have   ${\widehat{f}(\eta)\ge C \jp{\eta}^{-n/2-\gamma}}$. Then we have that   $f\notin W_{loc}^{\alpha,2}(\RR^n)$ if $\alpha \ge \gamma$.
\end{lemma}
\begin{proof}
We can always take a function $\psi \in C^\infty_c(\RR^n)$ such that $\widehat \psi(\xi) \ge 0$ in $\RR^n$ and $\widehat \psi(0)>0$ (for example, is enough to choose $\psi = \phi*\phi$ with $\phi \in C^\infty_c(\RR^n) $ radial and real, as in the definition of $g_\beta$). Then we can take an $0<\varepsilon<c/2$ small such that $\widehat{\psi}(\xi)$ is bounded below by a positive constant in $B_\varepsilon$. Hence if $ |\eta| \ge 2c$,
\begin{align*}
\widehat{\psi f}(\eta) &= \int_{\RR^n} \widehat \psi(\xi) \widehat{f}(\eta-\xi) \, d\xi \\
&\ge \int_{B_\varepsilon} \widehat \psi(\xi) \widehat{f}(\eta-\xi) \, d\xi \ge C\jp{\eta}^{-n/2-\gamma}.
\end{align*}
As a consequence we have that $\psi f\notin W^{\alpha,2}(\RR^n)$ for $\alpha\ge\gamma$, which implies that $f\notin W_{loc}^{\alpha,2}(\RR^n)$ by  definition of the local Sobolev spaces.
\end{proof}
\begin{proof}[Proof of Theorem $\ref{teo:Q2count}$]
 By Lemma \ref{lemma:gbeta} the function  $g_\beta$ satisfies all the conditions necessary to apply Lemma \ref{lemma.example.backs}, so for $\eta$ large we have
\begin{equation}\label{eq.max1}
 S(g_\beta)(\eta) \ge C \max\left( \jp{\eta}^{-\beta-n/2 -1} ,\jp{\eta}^{-2\beta-2 }\right ).
 \end{equation}
By (\ref{eq:Q2}), we have that
\begin{equation}\label{eq:Qg} \widehat{Q_{2}(g_\beta)}(\eta) = P(S_r(g_\beta))(\eta) + i\pi S(g_\beta)(\eta),
\end{equation}
and $\widehat{g}_\beta$ is real, so $P(S_r(g_\beta))$ and $S(g_\beta)$ are real functions of $\eta$ also. This means that if we assume $Q_{2}(g_\beta) \in W_{loc}^{\alpha,2}(\RR^n)$, we must have $\mathcal F^{-1} (S(g_\beta)) \in W_{loc}^{\alpha,2}(\RR^n) $, since there are no possible cancellations  between the real and imaginary  parts in (\ref{eq:Qg}). 

As a consequence of (\ref{eq.max1}), applying Lemma \ref{lemma.Wloc} with $f= \mathcal F^{-1} (S(g_\beta))$ we obtain that $\alpha$ must satisfy  simultaneously $\alpha< \beta+1$ and  $\alpha< 2\beta+ (n-4)/2$.

Hence, we have shown that for every $0<\beta<\infty$ there is a radial, real and compactly supported function $g_\beta$ such that $g_\beta \in W^{\gamma,2}$ if and only if $\gamma<\beta$, but we have that $ Q_2(g_\beta) \in W_{loc}^{\alpha,2}(\RR^n)$ only if $\alpha< \min(\beta + 1, 2\beta  -(n-4)/2)$. This enough to conclude the proof.
 \end{proof}


\section{Further Remarks}

In the introduction we have seen that there is a gap betwen the negative and positive results of recovery of singularities given in Theorems \ref{teo:main1} and   \ref{teo:main2}. This is a consequence of Theorems \ref{teo.Qj} and \ref{teo:Q2count}, where essentially the same gap is manifested in the results concerning the  $Q_2$ operator.  It  appears in the range   ${(n-4)/2 < \beta< (n-2)/2}$ (see for example figure \ref{fig.dim2y3} for the case $n=4$). What happens in this range is not known except for some partial results in dimension 2 and 3. In \cite[Proposition 3.1]{RV} a $1/2^-$ derivative gain for $\beta < 1/2$ is given    in dimension 3  using finer properties of the structure of the Ewald spheres. In this work, thanks to the trace theorem, we have not used any special properties of the spheres $\Gamma_r(\eta)$ in the estimates of the $Q_2$ operator. This suggest that there is an opportunity to improve the positive results in order to narrow this gap. Another possible strategy that we have already mentioned,  is to choose a weaker scale for measuring the regularity of the $Q_2(q)$ operator. This is the approach of \cite{BFRV13} in dimension 2, where they show that, if $\Lambda^\alpha(\RR^2)$ denotes the Hölder class and $q\in W^{\beta,2}(\RR^2)$, $\beta\ge 0$, then, modulo a $C^\infty$ function, $q-q_B \in \Lambda^\alpha(\RR^2)$  for every $\alpha<\beta$. This is a  $1^-$ derivative gain in the sense of integrability. A similar result holds also in dimension $n\ge 3$ and this will be the subject of a forthcoming work.

 A similar problem is what happens in the limiting case $\alpha =  \beta + 1$, when $\beta \ge (n-1)/2$. It is not difficult to show, modifying slightly the proof of lemma \ref{lemma.K.Q2}, that  there is a whole 1 derivative gain when $\beta > (n-1)/2$  for the spherical operator $S_r$. Unfortunately, it is not possible to say the same about the principal value operator, since in the estimate of the $P_{G,B}$ term in Proposition \ref{prop:genPV} is necessary to sacrifice an $\varepsilon$ of the regularity of the spherical operator (hence the final estimate for $\alpha'<\alpha$). Also, since this term involves cancellations,  it is difficult to determine if this is a limitation of the techniques, or if it is possible to construct a counterexample.

\section*{Aknowledgments}

I am very grateful to  my PhD advisors Alberto Ruiz  and Juan Antonio Barceló for their invaluable advice and constant support during the development of this work. 

The author was supported by Spanish government predoctoral grant BES-2015-074055 (project MTM2014-57769-C3-1-P).

\appendix

 \section{Some technical results} 
 
 \setcounter{equation}{0}
 
 \renewcommand{\theequation}{\Alph{section}.\arabic{equation}}
 
In this section we give the proofs of some technical results used throughout this work. 

\begin{lemma}[Trace theorem] \label{lemma.trazas2}
Let $\xi_1,\dots ,\xi_k \in \RR^n$, and $\mb r \in (0,\infty)^k $.  Assume that for every $\eta \in \RR^n$ the function $F(\xi_1,\xi_2,...,\xi_k,\eta)$ is $C^1$  in the first $k$ variables. Then if $\alpha_1,\dots,\alpha_k$ are  multi indices corresponding  to the variables $\xi_1,\dots ,\xi_k$ we have that
\begin{align*}
\nonumber &\int_{\Gamma_{\mb r}(\eta)}  |F(\xi_1,...,\xi_k,\eta)|^2 \,\, d\sigma_{\mb r}(\xi_1,\dots,\xi_k)   \\
 &\le C \sum_{0\le |\alpha_1|,\dots, |\alpha_k|\le 1} \int_{\RR^n} \dots  \int_{\RR^n} |\partial_{\xi_1}^{\alpha_1}\dots \partial_{\xi_k}^{\alpha_k} F(\xi_1,...,\xi_k,\eta)|^2 \,d \xi_1 \dots \, d \xi_k ,
\end{align*}
where the constant $C$ does not depend on $\eta$, or  $\mb r$.
\end{lemma} 
\begin{proof}
The general case follows inductively using that, as mentioned in the proof of Lemma \ref{lemma.K.Q2},  the estimate
\begin{equation} \label{eq:traza}
\int_{\SP_\rho} |f(x)|^2 \, d\sigma_\rho(x) \le  \int_{\RR^n} |f(x)|^2 \, dx + \int_{\RR^n} |\nabla f(x)|^2 \, dx ,
\end{equation}
  holds for any sphere  $\SP_\rho \subset  \RR^n$ of radius $\rho$ and (Lebesgue) measure $\sigma_\rho$, see { \cite[Proposition A.1]{fix}} for an elementary proof of this result.
\end{proof}

We now give the proof of the density lemma from which Lemma \ref{lemma:density}  follows directly.
\begin{lemma} \label{lemma:densityMain}
Let $X,Y$ be   Banach spaces, and let $D\subset Y$ be a dense subspace. Consider an operator  $T:D \longrightarrow X$ such that $T$ is the restriction to the diagonal of a multlinear operator of order $j$. That is, assume that there is some $Q: D\times, \dots,\times D \longrightarrow X$ multilinear such that for every $f\in D$, $T(f) = Q(f,\dots,f)$. Then, if for every $f\in D$
\begin{equation} \label{eq:XY}
 \lVert T(f) \rVert_X \le C \lVert f \rVert_Y^k,
\end{equation}  
 we have that there is a unique continuous extension of $T$ to the whole space $Y$, and it satisfies the estimate \eqref{eq:XY} for every $f\in Y$.
\end{lemma}
\begin{proof}
Let $\{g_i\}_{i \in \mathbb{N}}$, $g_i \in D$ for every $i \in \mathbb{N}$, be a Cauchy sequence in the $Y$ norm. To prove the proposition it is enough to show that then $\{T(g_i)\}_{i \in \mathbb{N}}$ is also a Cauchy sequence in $X$, since this implies that there must be a unique continuous extension of $T$ to the whole space $Y$.

Without loss of generality we can consider $Q$ symmetric, since otherwise we can take its symmetric part:
\[{Q}_{S}(f_1,\dots,f_j): = \frac{1}{j!} \sum_{\sigma} {Q}(f_{\sigma(1)},\dots,f_{\sigma(j)}),\]
where the sum is over all the permutations $\sigma$ of $j$ elements.  Therefore using the symmetry and the multilinearity  we have
\begin{multline} \label{eq:dens1}
T(g_k)-T(g_l) =   Q(g_k-g_l,g_k,\dots,g_k) + Q(g_l,g_k-g_l,g_k,\dots,g_k) \\
 +  Q(g_l,\dots,g_l,g_k-g_l) .
\end{multline}
We can now use a polarization identity for multilinear operators to express each of the previous terms as combinations of diagonal terms. See \cite{pol} for the explicit derivation of the identity:
\[j!Q(f_1,\dots,f_k)= \sum_{m=1}^j(-1)^{j-m} \sum_{J,|J|=m} T(S_J),\]
where the inner sum in the right hand side is over all distinct subsets $J \subset \{1,2,\dots,j\}$ of $m$ elements, and $S_J= \sum_{i\in J} f_i$.
Since each term in the last line of (\ref{eq:dens1}) can be treated in the same way, we illustrate only one case.
Let  $h>0$ be a (small) constant that we will choose later. Then  the polarization identity can be written in the following way
\begin{multline*}
 Q(g_l,g_k-g_l,g_k,\dots,g_k) = Q(hg_l,h^{-(j-1)}(g_k-g_l),hg_k,\dots,hg_k)\\
 = \frac{1}{j!}\sum_{0<a+b+c \le j} (-1)^{j-1-(a+b+c)} N(a,b,c)T( ah^{-(j-1)}(g_l-g_k) + h(bg_l+cg_k)),
\end{multline*}
where $ a,b, c$ are integers satisfying $0 \le a,b \le 1$, since $g_l$ and $g_l-g_k$ appear only once in the term we have chosen, and $0 \le c \le j-2$ since $g_k$ appears $j-2$ times (the integer coefficient $N(a,b,c)$ is just to account for repetitions).

Since $\{g_i\}$ is a Cauchy sequence, it is bounded, so $\lVert g_i \rVert_Y \le M$ for every $i \in \mathbb{N}$ and some constant $M>0$. Hence, taking the $X$ norm and using estimate (\ref{eq:estref}) we obtain
 \begin{align}
\nonumber  \lVert Q(g_l,&g_k-g_l,g_k,\dots,g_k) \rVert_X \le \\
 \nonumber &= \frac{1}{j!}\sum_{0<a+b+c \le j}  N(a,b,c)  \lVert  T( ah^{-(j-1)}(g_l-g_k) + h(bg_l+cg_k)) \rVert_X\\
\label{eq:dens2}& \le C(j)\left ( h^{-j(j-1)} \lVert g_l-g_k \rVert_{Y}^j + h^{j}M^j \right) <{\varepsilon}/{j},
\end{align}
for the following choices ($C(j)$ is some constant dependent only on $j$), 
\[h^j < \frac{\varepsilon}{2jC(j)M^j}, \, \, \, \text{and} \, \, \, \lVert g_l-g_k \rVert_Y < \frac{\varepsilon}{2jC(j)M^{j-1}}. \]
So, using (\ref{eq:dens2}) for each term in (\ref{eq:dens1}) we finally obtain
\[\lVert T(g_k)-T(g_l) \rVert_X <\varepsilon,\] 
which shows that $\{T(g_i)\}_{i\in \mathbb{N}}$ is a Cauchy sequence in $X$.
\end{proof}

\begin{proposition} \label{prop:smooth} Let $f \in \mathcal S$ and $1\le k \le j$. Then we have that
$$T_k(f)(\eta) = \prod_{i=1}^{k-1} (i\pi d_i + P_i) \widetilde S_{j, \mb r}(f)(\eta),$$
is a well defined function in $\mathcal S \left((0,\infty)^{j-k}\times \RR^n \right)$. Moreover, for $k \le  m \le j-1$
\begin{equation} \label{eq:comm}
\prod_{i=1}^{k-1} (i\pi d_i + P_i) \partial_{r_m} \widetilde S_{j, \mb r}(f)(\eta) = \partial_{r_m}\prod_{i=1}^{k-1} (i\pi d_i + P_i)  \widetilde S_{j, \mb r}(f)(\eta).
\end{equation}
\end{proposition}
\begin{proof}
We only sketch some of the main computations. Let's verify that $\widetilde S_{j\mb r}(f)$ is a function in $\mathcal S \left((0,\infty)^{j-1} \times \RR^n \right)$ (that is, the case $k=0$). First any derivative in the $\eta$ or $\mb r$ variables can be computed as in (\ref{eq:derivative1}) (see also  (\ref{eq:lastlemm})). The $|\eta|$ factors appearing in the expression of $\widetilde S_{j,\mb r}$ and its derivatives are non-smooth for $\eta = 0$, but this is not a problem since we have the smooth cut-off $\chi(\eta)$ which vanishes  in the origin. Essentially the estimate of each Schwartz class seminorm  can be reduced to the basic case
\begin{align*}
&\jp{\eta}^\gamma  \int_{\Gamma_{\mb r}(\eta)} |f(\eta -\xi_1)| \left(\prod_{i=1}^{j-2} |f(\xi_i - \xi_{i+1}) |\right) |f(\xi_{j-1})| \, d\sigma_{\mb r} \le \frac{C}{|\eta|^{n-1}} \dots\\ 
 &\int_{\Gamma_{\mb r}(\eta)} |f(\eta -\xi_1)|\jp{\eta-\xi_1}^{\gamma'} \left(\prod_{i=1}^{j-2} |f(\xi_i - \xi_{i+1})|\jp{\xi_i-\xi_{i+1}}^{\gamma'} \right) |f(\xi_{j-1})|\jp{\xi_{j-1}}^{\gamma'} d\sigma_{\mb r} \\
 &\le \norm{f\jp{\cdot}^{\gamma'}}_{\infty}^j,
\end{align*}
where $\gamma'= \gamma + n-1$. If instead of a weight $\jp{\eta}^\gamma $ we have $\jp{\mb r}^\gamma$, an analogous procedure can be followed using that $r_i = 2|\xi_i -\eta/2|/|\eta| \le C(1 + |\xi_i|)$ for $|\eta| >C_0$, that is, where the cut-off $\chi(\eta)$ does not vanish.

We give the following indications to  prove (\ref{eq:comm}) and that $T_k(f) \in \mathcal S \left((0,\infty)^{j-1-k}\times \RR^n \right)$ for $k>0$. Let $g \in \mathcal S( (0,\infty)^{k} )$ be a function of the variable $\mb r \in (0,\infty)^{k}$.  Is not very difficult to bound the principal value operators in the Schwartz class since we have the estimate
$$\norm{P_i(g) } \le C (\norm{\partial_{r_i} g}_\infty + \norm{\jp{r_i}^{\varepsilon}g }_\infty),$$
 for $\varepsilon>0$. This implies that $\partial_{r_m} P_i(g) = P_i(\partial_{r_m} g)$ since the limit that defines the derivative $\partial_{r_m}$ is continuous in the norms of the right hand side  (this is a consequence of the mean value theorem together with the fact that $g\in \mathcal \mathcal S( (0,\infty)^{k} )$ which means that all the derivatives are uniformly bounded). The same reasoning can be applied to control the partial derivatives in $\eta$ of $T_k(f)$.
\end{proof}

We state now Lemma \ref{lemma.LapFracL1}, used in the proof of Proposition \ref{prop:convergence}.

\begin{lemma} \label{lemma.LapFracL1}
Let $f,g\in C^\infty_c(\RR^n)$, then if $\beta \ge 0$ we have that
$$\norm{(-\Delta)^{\beta/2}(fg)}_{L^1} \le C(\beta) \norm{f}_{W^{\beta,2}} \norm{g}_{W^{\beta,2}} .$$
\end{lemma}
\begin{proof}
Assume first that $ 0 < \beta < 2 $ (the case of $\beta=0$ is trivial). Then we have the pointwise relation (it can be computed by hand using the principal value formula of the fractional Laplacian see for example \cite[p. 636]{soria})
\begin{align*}
(-\Delta)^{\beta/2} (fg)(x) = f(x)(-\Delta)^{\beta/2}(g)(x) &+ g(x) (-\Delta)^{\beta/2}(f)(x) \\
&+  \int_{\RR^n} \frac{(f(x)-f(y))(g(x)-g(y))}{|x-y|^{n+\beta}} \,dy,
\end{align*}
where we need $0 < \beta < 2 $ so that the singularity in the last integral can be controlled. Since by (\ref{eq:fracfourier}) we have that $\norm{(-\Delta)^{\beta/2} u}_{L^2} \le  \norm{u}_{W^{\beta,2}},$
 taking the $L^1(\RR^n)$ norm and applying Cauchy-Schwarz inequality to the first two terms we obtain
\begin{align*}
\norm{(-\Delta)^{\beta/2}(fg)}_{L^1} \le  2 &\norm{f}_{W^{\beta,2}} \norm{g}_{W^{\beta,2}} + \dots \\
 & \dots  \int_{\RR^n} \left| \int_{\RR^n} \frac{(f(x)-f(y))(g(x)-g(y))}{|x-y|^{n+\beta}} \,dy \right |\, dx.
\end{align*}
But the last can be bounded using Cauchy Schwarz and that
$$ \int_{\RR^n}  \int_{\RR^n} \frac{|f(x)-f(y)|^2}{|x-y|^{n+\beta}} \,dy \, dx \le C \norm{f}_{W^{\beta/2,2}}^2,$$
(in fact, the left hand side  is an equivalent norm for the homogeneous $\dot{W}^{\beta/2,2}$ when $0<\beta<2$, see \cite[Proposition 3.4]{valdinoci}). 

If we assume now that $\beta\ge 2$,  define $k := [\beta/2]$, that is the integer part, and  $\tilde \beta = \beta-2k$ so we have now $\tilde{\beta}\in [0,2)$, and
\begin{equation} \label{eq.frac1}
(-\Delta)^{\beta/2} (fg) = (-\Delta)^{\tilde{\beta}/2}(-\Delta)^k (fg).
\end{equation}
An integer power of the Laplacian is an homogeneous constant coefficient differential operator of order $2k$, and therefore if $a,b\in \NN^{n}$ we have
$$(-\Delta)^k (fg) = \sum_{|a|+|b|=2k} c_{a,b}\partial^a f \partial^b g,$$
where we are not interested in the particular values of the constants $c_{a,b}$.  Then, to bound the $L^1(\RR^n)$ norm of (\ref{eq.frac1}) we can apply the same arguments as before so we obtain
$$\norm{(-\Delta)^{\beta/2}(fg)}_{L^1} \le \sum_{|a|+|b|=2k} |c_{a,b}|\norm{\partial^a f}_{W^{{\tilde\beta},2}} \norm{\partial^b g}_{W^{{\tilde\beta},2}} \le C \norm{f}_{W^{\beta,2}} \norm{g}_{W^{\beta,2}}, $$
using that $\norm{\partial^a f}_{W^{{\tilde\beta},2}} \le \norm{ f}_{W^{{\tilde\beta} + |a|,2}} \le \norm{f}_{W^{\beta,2}}$ since $|a|\le 2k$.
\end{proof}

  \bibliographystyle{myplainurl}
  \bibliography{Backscattering}

\end{document}